\documentclass[11pt]{amsart} 
\usepackage{bbm, esint}
\usepackage[usenames, dvipsnames]{color}
\usepackage{amsmath,amssymb}
\usepackage{mfirstuc}
\usepackage{todonotes}
\usepackage{mathrsfs}
\usepackage{amssymb}
\usepackage{amsmath}
\usepackage{amsthm}
\usepackage{amsfonts}
\usepackage{color}
\usepackage{graphicx}
\usepackage[active]{srcltx}
\usepackage{tikz}
\usepackage[cp1252]{inputenc}

\usepackage{mathrsfs}
\usepackage{graphicx}

\usepackage[active]{srcltx}

\allowdisplaybreaks

\def\rr{{\mathbb R}}
\def\rn{{{\rr}^n}}

\def\nn{{\mathbb N}}

\def\cp{{\mathcal P}}
\def\cm{{\mathcal M}}

\def\fz{\infty}
\def\az{\alpha}

\def\dist{{\mathop\mathrm{\,dist\,}}}
\def\loc{{\mathop\mathrm{\,loc\,}}}

\def\lip{{\mathop\mathrm{\,Lip}}}

\def\lz{\lambda}
\def\dz{\delta}
\def\bdz{\Delta}
\def\ez{\epsilon}

\def\bz{\beta}

\def\esup{\mathop\mathrm{\,esssup\,}}

\newtheorem{thm}{Theorem}[section]
\newtheorem{lem}[thm]{Lemma}
\newtheorem{prop}[thm]{Proposition}
\newtheorem{rem}[thm]{Remark}
\newtheorem{defn}[thm]{Definition}
\numberwithin{equation}{section}

\renewcommand{\epsilon}{\varepsilon}
\begin{document}
\title[Jacobian determinant]{Jacobian determinants for (nonlinear)
 gradient of planar $\fz$-harmonic functions and applications}
 {\let\thefootnote\relax\footnotetext{H. Dong was partially supported by the Simons Foundation, grant no. 709545, a Simons fellowship, grant no. 007638, and the NSF under agreement DMS-2055244.
 F. Peng was supported by China Postdoctoral Science Foundation funded project
 (No. BX20220328). Y. Zhang was supported by the Chinese Academy of Science and NSFC grant No. 11688101.
Y. Zhou was supported by NSFC (No. 11871088 \& No.12025102) and by the Fundamental Research Funds for the Central Universities.
\endgraf}}
\author{Hongjie Dong, Fa Peng, Yi Ru-Ya Zhang, and Yuan Zhou}


\date{\today}
\arraycolsep=1pt
\allowdisplaybreaks
 \maketitle

\begin{center}
\begin{minipage}{13.5cm}\small
 \noindent{\bf Abstract.}\quad
In dimension 2,  we introduce a  distributional Jacobian determinant $\det DV_\bz(Dv)$ for  the nonlinear complex gradient $V_\bz(Dv)=
|Dv|^\bz(v_{x_1},-v_{x_2})$ for any $\beta>-1$,
whenever $v\in W^{1,2 }_\loc$  and $\bz |Dv|^{1+\bz}\in W^{1,2}_\loc$. This is
  new when $\bz\ne0$. 

\quad Given  any planar $\infty$-harmonic function $u$, we show that such distributional Jacobian determinant $\det DV_\bz(Du)$  is a nonnegative Radon measure with some  quantitative  local lower and upper bounds. We also give the following two applications.
\begin{enumerate}

\item[(i)]  Applying  this {result} with $\bz=0$,  we  develop an approach to build up
a Liouville theorem, which improves that of Savin  \cite{s}.  {Precisely},  if $u$ is $\fz$-harmonic functions in whole $\rr^2$ with
$$
\liminf_{R\to\fz}\inf_{c\in\rr}\frac1R\fint_{B(0,R)}|u(x)-c|\,dx<\fz,
$$
then $u=b+a\cdot x$ for some $b\in\rr$ and $ a\in\rr^2$.

\item[(ii)] Denoting by $u_p$ the $p$-harmonic function having the same nonconstant boundary condition as $u$,
  we show that $\det DV_\bz(Du_p) \to \det DV_\bz(Du)$ as $p\to\fz$ in the weak-$\star$ sense in the   space of Radon measure.
Recall that $V_\bz(Du_p)$ is always quasiregular mappings, but  $V_\bz(Du)$ is not in general.
\end{enumerate}



 \end{minipage}
 \end{center}

 \tableofcontents

\section{Introduction}

Let $\Omega$ be a  domain (connected open subset) in $ \rr^n$.
We say a function $u\in C^0(\Omega)$ is  $\fz$-harmonic if it is a viscosity solution to the $\fz$-Laplace equation
$$\Delta_\fz u=D^2uDu\cdot Du=0\quad\mbox{in}\ \Omega.$$
This equation  was derived by Aronsson  in 1960's as the  Euler-Lagrange equation for absolutely minimizers respect to the $L^\fz$-functional
$$
\mbox{$F(u,U)=\|\frac12|Du|^2\|_{L^\fz(U)}$ for domains $U\Subset\Omega$.}
$$
See \cite{a1,a2,a3,a4}.
For {a} probability interpretation (via Tug-of-War) of the $\fz$-Laplace equation, we refer  the reader  to \cite{pssw}.
Jensen \cite{j93} identified  absolute minimizers  with $\fz$-harmonic functions,
and moreover, built up their existence and uniqueness in bounded domains with continuous  boundary data. Since then, the regularity of $\fz$-harmonic functions has been a main issue in this field.  {Recall that} $\fz$-harmonic functions are always locally Lipschitz and hence differentiable almost everywhere.
{In view of} the $\fz$-harmonic function  $w=x_1^{4/3}-x^{4/3}_2$ in $\rr^n$ given by Aronsson \cite{a84},
it was conjectured in the literature  that  $\fz$-harmonic functions  have  $C^{1,1/3}$ and also $ W^{2,q}$ regularity with $q <3/2$.

Towards this conjecture, Crandall-Evans \cite{ce} first obtained a linear approximation property for
any $\fz$-harmonic function $u$: at each point $x$ and
for any sequence $\{r_j\}_{j\in\nn}$ converging to $0$, {there are} a subsequence $\{r_{j_k}\}_{k\in\nn}$  and also a vector $e$ depending on  $x$ and $\{r_{j_k}\}_{k\in\nn} $ such that
$$\lim_{k\to\fz}\sup_{z\in B(0,1)}\left|\frac{u(x+r_{j_k}z)-u(x)}{r_{j_k}} - e\cdot z  \right|=0 $$
and  $|e|=\lip u(x)$, where and { in the sequel}
  the pointwise Lipschitz constant of $u$ at   $x$  is defined {as}
$$\lip u(x)=\limsup_{x\ne y\to x}\frac{|u(y)-u(x)|}{|x-y|}.$$
The vector $e$ was then proved to be independent of the choice of subsequence, {which implies that} $u$ is differentiable at any point $x$. See Savin \cite{s} in dimension 2  based on  a planar topological argument
 and Evans-Smart \cite{es11a,es11b} in dimension $n\ge2$   via  some  PDE approach (flatness estimates).   In dimension $2$,   Savin \cite{s} further proved the $C^1$ regularity of $u$, and Evans-Savin \cite{es08}  {obtained} the $C^{1,\alpha}$-regularity of $u$ for some $0<\alpha<<1/3$.
Recently, it was proved in \cite{kzz} that
$|Du|^\alpha \in W^{1,2}$ for any $\alpha>0$, which is sharp as $\az\to 0$ as witted by  $x_1^{4/3}-x_2^{4/3}$.
Moreover, the distributional determinant of Hessian,  $-\det D^2u$,
 was proved in \cite{kzz} to be a  Radon measure  (in short $-\det D^2u\in \cm(\Omega)$) enjoying  the lower bound  $-\det D^2u\ge |D|Du||^2\,dx $ in measure  sense,
i.e.,
\begin{equation}  \label{xe1.1}
\int_\Omega-\det D^2u \psi\,dx\ge  \int_\Omega|D|Du | |^2\psi\,dx
 \quad \forall \ 0\le \psi\in C_c^0(\Omega),
 \end{equation}
and also the upper bound
\begin{equation*}
\int_{\frac12B}-\det D^2u\,dx\le C
\fint_{ B} |Du|^2\,dx \quad\forall\  B\Subset\Omega.
\end{equation*}
Recall that,  for any  $v\in W^{1,2}_\loc(\Omega)$, the distributional determinant $-\det D^2v$ is given by
\begin{align}\label{xe1.3}
-\int_\Omega\det D^2v\psi\,dx
&= \frac12\int_\Omega (D^2\psi Dv\cdot Dv) \,dx
-\frac{1}{ 2}
\int_\Omega|Dv|^{ 2} \bdz\psi \,dx \quad\forall\,\psi\in C_c^\fz(\Omega).
\end{align}

The main purpose of this paper is two-fold. First, via  the distributional determinant of Hessian we develop  a   new approach to build up a {gradient estimate and} Liouville theorem  for planar $\fz$-harmonic functions. See Theorem \ref{thm1.1} below.
Recall that Aronsson  \cite{a68} initiated the study of such Liouville theorems by proving that  planar $\infty$-harmonic functions of  $C^2(\rr^2)$ must be affine functions.
In the sequel, we denote by $\cp$  the collection of affine functions $P(x)=b+a\cdot x$    for some $b\in\rr$ and $a\in\rr^2$.
Evans \cite{e93} obtained an analogue result for   $\infty$-harmonic functions of $C^4(\rn)$ with $n\ge3$.
In all dimensions $n\ge2$, Crandall-Evans-Gariepy \cite{ceg} showed that any bounded  $\fz$-harmonic function   in  $\rr^n$   must be a constant, and also that  any  $\fz$-(sub)harmonic function  in $\rn$ bounded from above by some affine function $P$  must be  given by $P$.
In the plane, from the $C^1$-regularity and a compactness argument,  Savin \cite{s}  proved that any  $\fz$-harmonic function $u$ in $\rr^2$
with  the linear growth at $\fz$ (i.e., $|u(x)|\le C(1+|x|)$) must be an affine function.   However a high dimensional analogue is quite open.

We obtain the following {interior gradient estimate and} Liouville theorem in {the} plane, { the latter of} which improves that of Savin \cite{s} { mentioned above}.
 \begin{thm}\label{thm1.1}
$(i)$  {Let $u$ be an $\fz$-harmonic in a domain $\Omega\subset \rr^2$. Then we have
 $$|Du(x)|\le\frac{C }{r} \fint_{B(x, r)}| u|\,dy\quad\mbox{whenever $B(x, r)\subset\Omega$}  $$
 and hence
\begin{align*}
\|Du \|_{L^\fz(B(x,r))}\le  \frac C {r^{3}}\|u\|_{L^1(B(x,2r))}\quad \mbox{whenever $B(x,2r)\subset\Omega$}.
\end{align*}}

$(ii)$ Let $u $ be an $\fz$-harmonic function in $\rr^2$ with
\begin{align*}
\liminf_{R\to\fz}\inf_{c\in\rr}\frac1R\fint_{B(0,R)}|u(x)-c|\,dx<\fz.
\end{align*}
Then $u\in\cp$, i.e., $u(x)=u(0)+a\cdot x$ in $\rr^2$ for some vector $a\in\rr^2$.
\end{thm}

Our approach {of the proof of the Liouville theorem} is completely different from that of Savin \cite{s}.
The crucial point is that, given any  $\fz$-harmonic  function $u$ in a planar domain $\Omega$,

\begin{enumerate}
 \item[$\bullet$]
   in Lemma \ref{lem2.2}, we  derive the   identity
\begin{align*}
 \int_\Omega(-\det D^2u) u^2\psi\,dx&= - \int_\Omega |Du|^4\psi \,dx -\int_\Omega |Du|^2 (Du\cdot D\psi) u \,dx\\
&\quad +\frac12\int_\Omega u^2[ |Du|^2\Delta\psi -D^2\psi Du\cdot Du] \,dx\quad\forall\, \psi\in C_c^\fz(\Omega), \nonumber
\end{align*}
{ which implies Theorem \ref{thm1.1} $(i)$;}
\item[$\bullet$]
 in Proposition \ref{lem2.8}, we  obtain an upper bound which improves the one in \cite{kzz}:
$$\int_{\frac14 B}-\det D^2u\,dx\le C  \inf_{P\in\cp}\left\|\frac{u-P}r\right\|_{L^\fz( B)}\Big[|DP|+\|\frac{u-P}r\|_{L^\fz( B)} \Big]
\quad\forall\ B=B(x,r) \Subset\Omega.$$
\end{enumerate}
Both results have their own interests. These, together with \eqref{xe1.1} and some basic properties, allow  us to obtain  Theorem \ref{thm1.1} {$(ii)$}.
See Section \ref{sec2} for details.

Next, inspired by the limiting behaviors of planar $p$-harmonic functions (see the end of this section for details),
  we are interested in the Jacobian determinants
of the nonlinear complex gradient
$$
V_\bz(Du)=|Du|^{\beta}(u_{x_1},-u_{x_2})\quad\text{with}\,\bz>-1
$$
for planar $\fz$-harmonic functions.
In the special case $\bz=0$,   $\det DV_0(Du)=-\det D^2u$ is already defined by \eqref{xe1.3} as a distribution.
 However, in the general case $\bz\ne0$, since the Sobolev regularity of $|Du|^\bz Du $ is quite open,
there is no appropriate  definition for
$\det DV _\bz(Dv)=-\det D\big[|Du|^\bz Du\big]$   available in the literature.
In this paper, we find out  the following distributional definition,
which has its own interest.

\begin{defn}\label{def1.2}
Let $\Omega\subset\rr^2$ be a domain.  For any $\beta>-1$ and $v\in W^{1,2}_\loc(\Omega)$ satisfying $\beta |Dv|^{\bz+1}\in W^{1,2}_\loc(\Omega)$, we define
  $\det DV_\beta(Dv)$ as a  distribution  by
\begin{align}\label{xe1.5}
\int_\Omega\det DV_\beta(Dv)\psi\,dx
&:=-\frac12\int_\Omega|Dv|^{2\beta} (D^2\psi Dv\cdot Dv) \,dx
+\frac{1}{2\beta+2}
\int_\Omega|Dv|^{2\beta+2} \bdz\psi \,dx\\
&\ \quad-\frac{\beta}{\beta+1}
\int_\Omega[D |Dv| ^{\bz+1 }\cdot D v] (Dv\cdot D\psi)|Dv|^{\beta-1} \,dx\quad \forall\,   \psi\in C^\fz_c(\Omega). \nonumber
\end{align}
 \end{defn}

Compared to \eqref{xe1.3} in the case when $\bz=0$,   we  need the additional assumption  $ |Du|^{1+\bz}\in W^{1,2}_\loc(\Omega)$ in the case when $\bz\ne 0$.

Before  using Definition \ref{def1.2}, we must first verify  that it makes sense.
To be precise,
if $V_\bz(Dv)\in W^{1,2}_\loc(\Omega)$ a priori, 
then  we have a pointwise defined Jacobian determinant $\det DV _\bz(Dv)\in L^1_\loc(\Omega)$. On the other hand, since  $V_\bz(Dv)\in W^{1,2}_\loc(\Omega)$ implies
   $|Dv|^{1+\bz}\in W^{1,2}_\loc(\Omega)$, Definition \ref{def1.2} gives a  distributional Jacobian determinant $\det DV _\bz(Dv)$. One has to show the  coincidence between the pointwise definition and {the} distributional definition of  $\det DV _\bz(Dv)$.
In the case when $\bz=0$,  such coincidence is well known. Indeed,
 such coincidence holds  for smooth functions $v\in C^3(\Omega)$
directly via the divergence of structure of  $-\det D^2v$, and then  for  $v\in W^{2,2}_\loc(\Omega)$ via a standard approximation argument and linearity of $Dv$ and $D^2v$.
However, in the general case when $ \bz\ne0$,  due to several essential difficulties caused by
the nonlinear structure of $ V_\bz(Dv)$,  essentially new ideas are required to get such coincidence.
Eventually we are able to prove such coincidence via
\begin{enumerate}
\item[$\bullet$]
a  nonlinear second order estimate to inhomogeneous $(\bz+2)$-Laplace equations  by Cianchi-Mazya \cite{cm}.
\item[$\bullet$] a fundamental structural identity and a divergence structure for $-\det D[(|Dv|^2+\ez)^{{\bz/2}}Dv]$ with $v\in C^\fz(\Omega)$ and $\ez>0$. See Lemmas \ref{lem3.7} and \ref{lem3.2}.
\item[$\bullet$] a divergence  structure of $-\det D^2v$ with $v\in C^\fz$ and its connection with $\Delta_\fz v$. See Lemmas \ref{lem2.1}  and \ref{lem3.6}.
\end{enumerate}
See Section \ref{sec3} for the proof.  In Section \ref{sec4}, we present some useful properties {of}  the distributional $ \det DV_\beta( Dv)$.

For any planar $\fz$-harmonic function $u$ and any $\bz>-1$, since
  $|Du|^{\bz+1}\in W^{1,2}_\loc$ as proved in \cite{kzz},  the
 distributional Jacobian determinant $\det DV _\bz(Dv)$ is  defined by \eqref{xe1.5}. Recalling that
  $D|Du|^{\bz+1}\cdot Du=0$  almost everywhere (see \cite{kzz}), for any  $\psi\in C^\fz_c(\Omega)$ one has
\begin{align}\label{xe1.6}
\int_\Omega\det DV_\beta(Du)\psi\,dx
& =-\frac12\int_\Omega |Du|^{2\beta} (D^2\psi Du\cdot Du) \,dx
+\frac{1}{2\bz+2}
\int_\Omega|Du|^{2\beta+2} \bdz\psi \,dx.
\end{align}
We then obtain the following result.

\begin{thm}\label{thm1.3}
Let $\Omega\subset\rr^2$ be a domain, and let    $\beta>-1$.    For any $\fz$-harmonic function $u$ in $\Omega$, we have  $\det DV_\beta(Du )\in \cm(\Omega)$
with the lower bound
   \begin{align}\label{xe1.7}\int_\Omega\det DV_\beta(Du )\psi\,dx\ge \frac1{\bz+1}\int_\Omega|D|Du |^{\bz+1}|^2\psi\,dx
 \quad \forall\, 0\le \psi\in C_c^0(\Omega)
\end{align}
and  the upper bound
   \begin{align}\label{xe1.8}\int_{\frac12 B} \det DV_\beta(Du )\,dx\le C\big(1+\frac 1 {\beta+1}\big)  \fint_{B}|Du|^{2+2\bz}\,dx \quad \forall\, B\Subset\Omega,
   \end{align}
 where $C>0$ is a universal constant.
\end{thm}

We prove Theorem \ref{thm1.3}  in Section \ref{sec5}.
By using Lemmas \ref{lem2.1}, \ref{lem3.6}, \ref{lem3.7}, 4.1, and 4.2, we  build up some analogue lower and upper bounds
for exponential $e^{\frac1{2\ez}|\xi|^2}$-harmonic function $u^\ez$ in $U\Subset\Omega$,   which are uniform in $\ez\in(0,1)$. {As a consequence}, we conclude that  $ \det DV _\bz(Du^\ez)\in \cm(U)$ and
  $ \det DV _\bz(Du^\ez)\to  \det DV _\bz(Du )$ in the weak-$\star$ sense in $\cm(U)$.

{Our original motivation of Definition  \ref{def1.2} and Theorem \ref{thm1.3} is to study
planar $p$-harmonic functions $u_p$ and their limiting behavior as $p\to\fz$.} A function $v\in W^{1,p}_\loc(\Omega)$ is called $p$-harmonic if it is a weak solution to
$$ -\Delta_pv={\rm div}(|Dv|^{p-2}Dv)=0\quad\mbox{in $\Omega$.}$$
See \cite{ps} for { a} probability interpretation by using Tug-of-War with noise.
We refer the reader to Iwaniec-Manfredi \cite{im89} and Aronsson \cite{a89} for the   $C^{k,\alpha}$ and $W^{k+2,q}_\loc$-regularity of $u_p$ with optimal $k$, $\alpha$, and $q$.
In the literature,  the interior regularity of $p$-harmonic functions  in any dimension has been extensively studied. See \cite{u68,u77,e82,d83,l83,bi83,t84,m88,mw88,dfzz,ss}.

Moreover, for each $\bz>-1$ and $1<p<\fz$,  the nonlinear complex gradient $V_\bz(Du_p)$ was well-known to be  a quasiregular mapping. Precisely,
\begin{equation} \label{xe1.9}
\mbox{$V_\bz(Du_p)\in W^{1,2}_\loc(\Omega)$ and \ }|DV_\bz(Du_p)|^2\le [K(p,\beta)+\frac1{K(p,\beta)}]\det D V_\bz(Du_p)
 \quad \mbox{a.\,e.\ in $\Omega$},
\end{equation}
where
$$   K(p,\bz)=\max\left\{ \frac{p-1}{\bz+1},\frac{\bz+1}{p-1},
\bz+1,\frac1{\bz+1} \right\},$$
{which} leads to a pointwise defined  Jacobian determinant  $ \det DV _\bz(Du_p)\in L^1_\loc(\Omega)$.
We refer the reader to Bojarski-Iwaniec  \cite{bi87},  Manfredi \cite{m88}, Iwaniec-Manfredi \cite{im89}, and Aronsson \cite{a89}. {See also}  Lemma \ref{lema.3} Remark \ref{remA.4}   for the sharpness of the constant  in \eqref{xe1.9} {by using} some construction in \cite{im89}.

In this paper, we obtain the following lower and upper bounds, where  the constant is uniform in $p\ge 2$ and hence  is quite different from above quasiregular properties.

\begin{thm}\label{thm1.4}
Let  $\beta>-1$ and $1<p<\fz$. For any $p$-harmonic functions $u_p$  in {a} given planar domain $\Omega$ one has the lower bound
   \begin{align*}
   \det DV_\beta(Du_p )
& \ge \frac{\min\{1,p-1\}}{\bz+1}
  |D  |D u_p|^{\bz+1}  |^2
\quad{\rm a. \, e.},
  \end{align*}
and  upper bound
   \begin{align*} \int_{\frac12B}
   \det DV_\beta(Du_p )\,dx\le C
   \big[1+  \frac1{1+\bz}+\frac1{p-1}\frac{\bz^2}{\bz+1}\big] \fint_{B}|Du_p|^{2+2\bz}\,dx \quad\forall\ \text{ball} \ B\Subset\Omega.
\end{align*}
\end{thm}
Theorem  \ref{thm1.4} will be proved in Section \ref{sec6} based on   the approach in \cite{kzz},
Lemma \ref{lem3.7}, Lemma \ref{lem2.1}, Lemma \ref{lem3.6}, and Lemma \ref{lem3.2}.
Note that the case $\bz=0$ was already given in a similar but simpler way by Lindgren-Lindqvist \cite{ll} via the approach in \cite{kzz}.

If  $p$-harmonic functions $u_p$ in a bounded smooth domain $\Omega$ have the same boundary data $g\in C^{0,1}(\partial\Omega)$, then by using the variational approach and Jensen \cite{j93}, there is  a function $u_\fz\in C^{0,1}(\overline \Omega)$, which is the unique $\fz$-harmonic function with boundary data $g$, so that $u_p\to u_\fz$ in $C^{0,\alpha}(\overline\Omega)$ for any $\alpha\in(0,1)$ and  weakly in $W^{1,q}_\loc(\Omega)$ for any $1<q<\fz$.

Recently, based on the approach in \cite{kzz} and Theorem {\ref{thm1.4}} with $\bz=0$, Lindgren-Lindqvist \cite{ll}  deduced that     $ u_p\to  u_\fz $ in $W^{1,q}_\loc(\Omega)$ for any $1<q<\fz$.
However, even  though  we already know that
$u_\fz\in C^1(\Omega)$ and $ C^{1,\alpha}(\Omega)$ for some $0<\alpha< 1/3$ (see   \cite{s,es08}),
 it remains open whether   $ u_p\to  u_\fz $ in   either  $C^1(\Omega)$ or $C^{1,\alpha}(\Omega)$.
We also observe that from $p$-harmonic to $\fz$-harmonic functions, the best possible regularity has a huge jump. For example  one always has  $u_p\in W^{3,1}_\loc$, but does
 not necessarily  has $u_\fz\in W^{2,3/2}_\loc$.

Because $Du_p\to Du_\fz$ in $L^q_\loc(\Omega)$ with $1<q<\fz$, for $\bz>-1$, the limit of   mappings $V_\bz(Du_\fz)$  as $p\to\fz$ is naturally expected to be  the mapping $V_\bz(Du_\fz)$ in certain sense.
However, since  $K(p,\bz)\to \fz$ as $p\to\fz$, one cannot expect that $V_\bz(Du_\fz)$ is a quasiregular. Indeed,  we do not necessarily have $ V_\bz(Du_\fz)\in W^{1,2}_\loc(\Omega)$   as witted by Aronsson's $\fz$-harmonic function $x_1^{4/3}-x_2^{4/3}$.
Moreover, the   $ W^{1,1}_\loc$-regularity
 of $V_\beta(Du_\fz)$
 is quite open and very difficult even in the case special when $\bz=0$,  and
 a pointwise Jacobian determinant $\det D V_\beta(Du_\fz) $ is unavailable.

Instead of the pointwise one, in the case when $\bz=0$,
 we already have the distributional  Jacobian determinant $\det DV _0(Du_\fz)=-\det D^2u_\fz$ as in \eqref{xe1.3}.
Because $Du_p\to Du_\fz$ in $L^2_\loc(\Omega)$,  one has
$\det DV _0(Du_p)\to \det DV _0(Du_\fz)$ in the sense of distributions.
Since $\det DV _0(Du_\fz)$ was proved to  be a Radon measure in \cite{kzz}, it is naturally
 expected that $\det DV _0(Du_p)\to \det DV _0(Du_\fz)$  in the  weak-$\star$ sense  in $\cm(\Omega)$.

In the case when $0\neq \beta>-1$,  it is a basic question
whether the limit $\lim_{p\to\fz}\det D V_\beta(Du_p) $  exists in certain sense.
If so, it is   expected  to be given by  $\det D V_\beta(Du_\fz) $.     However, unlike the case $\bz=0$, a distributional definition for $\det D V_\beta(Du_\fz) $ is unavailable
in the literature.
This leads us to introduce the distributional  $\det D V_\beta(Dv) $ as in Definition \ref{def1.2} and   build up Theorem \ref{thm1.5} below, which answer these questions.

\begin{thm}\label{thm1.5}
Given any  $\beta>-1$,
as $p\to\fz$ one has that   $V_\beta (Du_p)\to V_\beta (Du_\fz)$ in $ L^q_\loc(\Omega)$ for any $q>1$,
and also  that
$\det DV_\beta (Du_p)\to \det DV_\beta(Du_\fz)$ in the weak-$\star$ sense in $\cm (\Omega)$, that is,
  $$\int_\Omega \det DV_\beta(Du_\fz)\psi\,dx=\lim_{p\to\fz}\int_\Omega \det DV_\beta (Du_p)\psi\,dx \quad\forall\, \psi\in C_c^0(\Omega).$$
\end{thm}

Theorem \ref{thm1.5} follows from Theorems \ref{thm1.3} and \ref{thm1.4}, and the convergence $Du_p\to Du_\fz$ in $L^q_\loc(\Omega)$ with $1<q<\fz$ as given in \cite{ll}. See Section \ref{sec6} for details.

 \section{Proof of Theorem  \ref{thm1.1}}
                \label{sec2}
 We  begin with the following divergence structure of $-\det D^2v$ for $v\in C^\fz$.
\begin{lem}\label{lem2.1}
 For any $v\in C^\fz(\Omega)$,  one has
 \begin{align}\label{xe2.1}-\det D^2v=\frac12\big[|D^2v|^2-(\Delta v)^2\big]=
\frac12 {\rm div}(D^2vDv-\Delta v Dv)=\frac12\big[\Delta(|Dv|^2)-(v_{x_i}v_{x_j})_{x_ix_j}\big].
\end{align}
Consequently,  \begin{align}\label{xe2.2}
\int_\Omega
-\det D^2v\psi\,dx&=-
\frac12 \int_\Omega [D^2vDv-\Delta v Dv]\cdot D\psi\,dx\\
&=
\frac12 \int_\Omega [|Dv|^2\Delta\psi-D^2\psi Dv \cdot D v]\,dx\quad\forall \psi\in C_c^\fz(\Omega).\nonumber\end{align}
\end{lem}

From  this we deduce the following formula for $(-\det D^2u) u^2$ for $\infty$-harmonic functions $u$
in $\Omega\subset\rr^2$.
Since $u$   always enjoys  locally Lipschitz regularity (see Jessen \cite{j93}),
by Rademacher's theorem, $u$ is  differentiable   almost everywhere, and hence for almost all $x\in\Omega$, the derivative $Du(x)$ exists and $\lip u(x)=|Du(x)|$.  Moreover,
$$\|Du\|_{L^\fz(U)}=\sup_{x\in U}\lip u(x)\quad\mbox{ for any domain}\ U\subset\Omega.$$
{As such}, when there is no confusion, we {slightly abuse the notation by writing}  $\lip u $ {instead of} $|Du|$.
We remark that even {though} $u$ {is} known to be everywhere differentiable (see Evans-Smart \cite{es11a,es11b}) and also $C^1$-regular (see  Savin \cite{s}),
all the results in \cite{kzz} and also all our results and proofs below do not rely on {either the} everywhere differentiability  {or the   $C^1$-regularity of $u$.

\begin{lem}\label{lem2.2} If $u$ is $\fz$-harmonic  in $\Omega$, then for any $\psi\in C_c^\fz(\Omega)$ one has
\begin{align} \label{xe2.3}
& \int_U(-\det D^2u )u^2\psi\,dx  + \int_U |Du|^4\psi \,dx\\
&\quad =-\int_U |Du|^2 (Du\cdot D\psi) u \,dx +\frac12\int_U u^2[ |Du|^2\Delta\psi -D^2\psi Du\cdot Du] \,dx. \nonumber
\end{align}
\end{lem}

\begin{proof} {\bf Step 1.}
Given  any $v\in C^2(U)$ with $U\Subset\Omega$,
we show that
\begin{align}\label{ey.3}
 \int_U(-\det D^2v)v^2\psi\,dx  &= - \frac32 \int_U   (D|Dv|^2\cdot Dv) v  \psi\,dx-\int_U |Dv|^4\psi \,dx\\
&\quad-\int_U |Dv|^2 (Dv\cdot D\psi) v \,dx+\frac12\int_Uv^2\Big[ |Dv|^2\Delta\psi -D^2\psi Dv\cdot Dv\Big] \,dx.\nonumber
\end{align}

Indeed, by \eqref{xe2.1}  one has
\begin{align*}
 \int_U(-\det D^2v) v^2\psi\,dx&=-  \int_U   D^2v Dv \cdot  \big[vDv   \psi +\frac12 v^2 D\psi\big]\,dx+   \int_U  \Delta v Dv\cdot  \big[vDv  \psi+\frac12 v^2  D\psi\big]\,dx.
\end{align*}
Note that
\begin{align*}
 &-  \int_U   D^2v Dv \cdot  \big[vDv  \psi+\frac12 v^2 D\psi\big]\,dx =- \frac12 \int_U   (D|Dv|^2\cdot Dv) v  \psi\,dx-\frac14 \int_U  (D|Dv|^2\cdot D\psi)v^2  \,dx,
\end{align*}
where,  using integration by parts,
\begin{align*}
 -\frac14 \int_U  (D|Dv|^2\cdot D\psi)v^2  \,dx
 &=   \int_U \Big[\frac14|Dv|^2\Delta\psi v^2 +  \frac12 |Dv|^2 (Dv\cdot D\psi) v\Big]\,dx.
\end{align*}
Moreover,  integration by parts yields
\begin{align*}
& \int_U  \Delta v Dv\cdot  \big[vDv  \psi +\frac12 v^2  D\psi\big]\,dx\\
 &=  -  \int_U   D v \cdot D(  |Dv |^2 v \psi+\frac12 v^2 Dv\cdot D\psi )\,dx\\
 &=  -
 \int_U     (D|Dv|^2\cdot Dv)   v\psi\,dx  -\int_U |Dv|^4\psi \,dx-\frac14 \int_U v^2   D|Dv|^2\cdot D\psi\,dx \\
 &\quad-\int_U[2v |Dv|^2Dv\cdot D\psi+
 \frac12 v^2 D^2\psi Dv\cdot Dv \,dx\\
 &=  -\int_U     (D|Dv|^2\cdot Dv)   v\psi\,dx  -\int_U |Dv|^4\psi \,dx+\frac14 \int_U v^2   |Dv|^2 \Delta\psi\,dx  \\
 &\ \quad -\int_U\big[\frac 3 2 v |Dv|^2Dv\cdot D\psi+
 \frac12 v^2 D^2\psi Dv\cdot Dv\big] \,dx.
\end{align*}
Combining these, we conclude \eqref{ey.3}.

{\bf Step 2.}  Given any smooth subdomain $U\Subset\Omega$, let $u^\ez\in C^\fz(U)\cap C^0(\overline U)$ be the solution to
$$\mbox{ ${\rm div}(e^{\frac1{2\ez} |Du^\ez|^2}Du^\ez)= \frac 1 \ez e^{\frac1{2\ez} |Du^\ez|^2}(\Delta_\fz u^\ez+\ez\Delta u^\ez)=0$ with $u^\ez=u$ on $\partial U$. }  $$

Given any $\psi\in C_c^\fz(U)$, we observe that
$$
 \int_U(-\det D^2u) u^2\psi\,dx  = \lim_{\ez\to0}\int_U(-\det D^2u^\ez) (u^\ez)^2 \psi\,dx.
$$
Indeed,
\begin{align*}
&\big|\int_U(-\det D^2u ) u^2\psi\,dx-\int_U(-\det D^2u^\ez) (u^\ez)^22\psi\,dx\big|\\
&\quad\le   \big|\int_U(-\det D^2u) u^2\psi\,dx-\int_U(-\det D^2u^\ez) u^2 \psi\,dx \big|
+ \big|\int_U(-\det D^2u^\ez) [u^2-(u^\ez)^2]\psi\,dx  \big|.
\end{align*}
By the definition of $-\det D^2u$ (see \cite[Theorem 1.5]{kzz} and also the proof of Theorem \ref{thm1.3} below), the first term goes to $0$ as $\ez\to0$.  The second term is bounded by
 $$
 \|\det D^2u^\ez\|_{L^1({\rm supp}\psi)}\|(u^\ez)^2-u^2\|_{L^\fz({\rm supp}\psi)}.
 $$
By using the fact
$-\det D^2u^\ez\in L^1_\loc(U)$ uniformly in small $\ez>0$  (see \cite[Theorem 1.5]{kzz} and also the proof of Theorem \ref{thm1.3} below), the second term also goes to  $0$ as $\ez\to0$.

 Applying \eqref{ey.3} proved in Step 1 to $u^\ez$, one has
 \begin{align*}
&\int_U(-\det D^2u^\ez) (u^\ez)^2\psi\,dx \\
&= - \frac32 \int_U   (D|Du^\ez|^2\cdot Du^\ez) u^\ez \psi\,dx-\int_U |Du^\ez|^4\psi \,dx\\
&\quad-\int_U |Du^\ez|^2 (Du^\ez\cdot D\psi) u^\ez \,dx+\frac12\int_U (u^\ez)^2\big[ |Du^\ez|^2\Delta\psi -D^2\psi Du^\ez\cdot Du^\ez\big] \,dx.
\end{align*}
Since $D|Du^\ez|^2 \to D|Du |^2$ weakly in $L^2_\loc(U)$  and $u^\ez\to u$ in $W^{1,q}_\loc(U)$ for any $1<q<\fz$ (see \cite{kzz}), sending $\ez\to0$  above and noting
$D|Du|^2\cdot Du=0$   one has desired identity \eqref{xe2.3}.
\end{proof}

Consequently, one has
\begin{lem} \label{lem2.3} If $u$ is $\fz$-harmonic  in $\Omega$,
then
\begin{align}\label{ey.1} \int_{B(x,r)}(-\det D^2u) u^2\,dx+  \|Du\|^4_{L^4(B(x,r))}\le  \frac C{r^4} \| u \|^4_{L^4(B(x,2r))}\quad \mbox{whenever $B(x,2r)\subset \Omega$. }
\end{align}
In particular, one has
\begin{align}\label{ey.2} \|Du\|_{L^4(B(x,r))}\le \frac C{r} \| u \|_{L^4(B(x,2r))}\quad \mbox{whenever $B(x,2r)\subset \Omega$. }\end{align}
 \end{lem}

 \begin{proof}
 Let $\phi\in C^\fz_c(B(x,2r))$ be a cut-off function {satisfying}
 \begin{align}\label{cut-off}
 \mbox{ $\phi=1$ in $B(x,r)$, $0\le \phi\le 1$ in $B(x,2r)$ and $|D\phi|^2+|D^2\phi|\le\frac C{r^2}$   in $B(x,2r)$.}
 \end{align}
Taking $\psi=\phi^4$ in \eqref{xe2.3},  one has
 \begin{align*}
&\int_{B(x,2r)}(-\det D^2u ) u^2\phi^4\,dx+ \int_{B(x,2r)} |Du|^4\phi ^4\,dx\\
&\quad= -\int_{B(x,2r)} |Du|^2 (Du\cdot D\phi^4) u \,dx +\frac12\int_{B(x,2r)} u^2\big[ |Du|^2\Delta\phi^4 -D^2\phi^4 Du\cdot Du\big] \,dx.
\end{align*}
By Young's inequality,  the right hand side of the above identity is bounded
 \begin{align*}
   \frac12 \int_{B(x,2r)} |Du|^4\phi ^4\,dx+
 C\int_{B(x,2r)} |u|^4\big[|D\phi|^4+\phi^2|D^2\phi|^2 \big]\,dx.
\end{align*}
Thus,
 \begin{align*}
 \int_{B(x,2r)}(-\det D^2u) u^2\phi^4\,dx+ \int_{B(x,2r)} |Du|^4\phi ^4\,dx\le
 C\int_{B(x,2r)} |u|^4\big[|D\phi|^4+\phi^2|D^2\phi|^2 \big]\,dx,
\end{align*}
which together with \eqref{cut-off}  gives the desired  \eqref{ey.1}.  Finally, since  $-\det  D^2u\ge0$  (see \cite{kzz}),  \eqref{ey.2} follows from \eqref{ey.1}.
 \end{proof}

From Lemma \ref{lem2.3} and some basic properties of $\fz$-harmonic functions we are able to get the following gradient estimate {in Theorem \ref{thm1.1}  $(i)$}.
\begin{proof}[Proof of Theorem \ref{thm1.1} $(i)$]
 Since $u-c$ is also $\fz$-harmonic for any constant $c$, by Lemma \ref{lem2.3},
$$\|Du\|_{L^4(B(x,r))}\le  \frac C r \| u-c\|_{L^4(B(x,2r))}\quad\forall c\in\rr $$
whenever $B(x,2r)\subset \Omega$.
Note that
$$ \| u-c\|_{L^4(B(x,2r))}\le  C \| u-c\|_{L^1(B(x,2r))}^{1/4}\|u-c\|^{3/4}_{L^\fz(B(x,2r))}.$$
Taking $c$ as the average  of $u$ in $B(x,2r)$, and applying the Sobolev-Poincar\'e inequality   one has
$$
\|u-c\|_{L^\fz(B(x,2r))}\le C r^{1/2}\|Du\|_{L^4(B(x,2r))}$$
and hence
$$
\|Du\|_{L^4(B(x,r))}\le  C r^{-5/8} \| u\|_{L^1(B(x,2r))}^{1/4}   \|Du\|_{L^4(B(x,2r))}^{3/4}.
$$
This and Young's inequality yield that, for any $\ez\in(0,1)$, one has
$$
r^{5/2} \|Du\|_{L^4(B(x,r))}\le\ez r^{5/2}\|Du\|_{L^4(B(x,2r))}+
  C_\varepsilon   \| u\|_{L^1(B(x,2r))} .
  $$
Via a standard iteration argument  (see, for instance, \cite[pp. 80--82]{Gia}), we conclude that
$$  r^{5/2} \|Du\|_{L^4(B(0,r))}\le
  C   \| u\|_{L^1(B(0,2r))} .$$
Applying  Morrey's inequality,
for any ball $ B(x,2r)\subset \Omega$ and  any$ y\in \Omega$ with $ |x-y|=  r $, by the above one has
$$|
u(x)-u(y)|\le |x-y|^{1/2}\|Du\|_{L^4(B(x,r))} \le   \frac{C}{r^{3}} |x-y|\| u\|_{L^1(B(x,2r))} .$$
One obtains
$$u(y) - \frac{C}{r^{3}}  |x-y|\| u\|_{L^1(B(x,2r))}\le u(x)
\le u(y) +  \frac{C}{r^{3}} |x-y|\| u\|_{L^1(B(x,2r))} $$
in $\partial (B(x,r)\setminus \{x\})$  and then, by the comparison property with cones,      in $B(x,r)$. Since $u$ is differentiable at $x$  (see \cite{es11b}), this yields that
$$
|Du(x)|\le\frac{C }{r} \fint_{B(x,2r)}| u|\,dz.
$$
Theorem \ref{thm1.1} $(i)$ is proved.
\end{proof}

\begin{rem} 
\rm
Let $u$ be an $\fz$-harmonic function in $\Omega\subset\rr^2$.

$(i)$   It was shown by Lindqvist and Manfredi \cite{lM} that
$$|Du(x)|\le \frac{2}r\|u\|_{L^\fz(B(x,r))}\quad \mbox{in $B(x,r) \subset\Omega$}.$$
{Theorem \ref{thm1.1} $(i)$ shows} that $2\|u\|_{L^\fz(B(x,r))}$  can be relaxed to the average  $C \fint_{B(x,r)}|u|\,dx$.

$(ii)$ Via  directly working on  $-\det D^2u^\ez$  for $e^{\frac1{2\ez}|\xi|^2}$-harmonic function $ u^\ez$ and some tedious calculation,
 it was  \cite[Theorem 1.5]{kzz} that,  for any $p>2$,
$$\|Du\|_{L^p(B(x,r))}\le \frac {C(p)} r \| u \|_{L^p(B(x,2r))}\quad  \mbox{ whenever $B(x,2r)\Subset \Omega$.} $$
In the case when $p= 4$,   we derived  a formula for $(-\det D^2u) u^2$ in Lemma \ref{lem2.2}, {simplifying} the argument and calculations in \cite{kzz}. See Lemma
 \ref{lem2.3} and its proof.
\end{rem}

The following is crucial to prove Theorem \ref{thm1.1} $(ii)$.  Recall that $\cp=\{b+a\cdot x|b\in\rr, a\in\rr^2\}$.

 \begin{prop}\label{lem2.8} If $u\in C^0(\Omega)$ is an $\fz$-harmonic function in a planar domain $\Omega$,
 then
$$\int_{\frac14 B}-\det D^2u\,dx\le C  \inf_{P\in\cp}\|\frac{u-P}r\|_{L^\fz( B)}\Big[|DP|+\|\frac{u-P}r\|_{L^\fz( B)} \Big]
\quad\forall\ B=B(x,r) \Subset\Omega.$$
\end{prop}
To prove this proposition, we need the following, which was proved
 by the comparison property with cones (see \cite[Section 2]{ceg}).  For reader's convenience we  briefly recall their proof  below.

\begin{lem}\label{lem2.6} If $u$ is $\fz$-harmonic  in $B(0,  2r)$, then
\begin{align}\label{xe2.5}
\|Du \|_{L^\fz(B(0,r))}\le \inf_{P\in \cp}\left[|DP|+2\|\frac{u-P}r\|_{L^\fz(B(0, 2r))}\right].
\end{align}

\end{lem}

\begin{proof}
Given any $P(x)=b+a\cdot x\in\cp$, write
 $\lz=\|u-P\|_{L^\fz(B(0,2r))}  .$
Given any $x\in B(0,r  )$, for any  $y\in\rr^2$ with $|x-y|=r $, we have $y\in B(0,2r)$ and
 $$|  u ( y) -u(x)-a\cdot(y-x)|\le |u(x)-b-a\cdot x|+ |u(y)-b-a\cdot y|\le 2\lz   ,$$
which implies that
 $$  u ( y) \le  u(x)+a\cdot(y-x)+2\lz \le u(x)+(|a|+2\frac \lz r)|x-y| $$
 and
  $$  u ( y) \ge  u(x)+a\cdot(y-x)-2\lz \ge u(x) -(|a|+2\frac \lz r)|x-y|.$$
Applying the comparison property with cones, one has
$$
|u ( y)-  u(x)| \le  (|a|+2\frac \lz r)|x-y|\quad \forall y\in B(x, r ),$$
which implies that
$|Du(x)|\le |a|+2\lz/r$, so \eqref{xe2.5} follows.
\end{proof}

The following property was observed in \cite{es11b}.
For reader's convenience, we give the proof below.

\begin{lem}\label{lem2.7} If $u$ is $\fz$-harmonic  in $B(0,2r)$, then  for any  $P\in\cp$ we have
\begin{align*}
\fint_{B(0, r)}|Du-DP|^2\,dx\le   20  \|\frac{u-P}r\|_{L^\fz(B(0,2r))}\Big[|DP|+\|\frac{u-P}r\|_{L^\fz(B(0,2r))} \Big].
\end{align*}
\end{lem}
\begin{proof}
By considering $u(rx) /r$ and $P(rx)/r\in \cp$, we may assume that $r=1$.
Write  $\lz=\|u-P\|_{L^\fz(B(0,2))}$.
If $\lz\ge |DP|$, then by Lemma \ref{lem2.6} one has
$$\fint_{B(0, 1)}|Du-DP|^2\,dx\le \fint_{B(0, 1)}2\big[|Du|^2+|DP|^2\big]\,dx\le  2[(|DP|+2\lz)^2+|DP|^2]\le 20\lz^2.$$
Below assume that $0<\lz<|DP|$.
Set  $v= u/|DP|$, $Q=P/|DP|$ and $\mu=\lz/|DP|=\|v-Q\|_{L^\fz(B(0,2))}<1$. Then $|DQ|=1$. Up to a rotation we may assume   $DQ=e_2$.
It then suffices to show
\begin{align}\label{xe2.7}\fint_{B(0,1)}|Dv-e_2|^2\,dx \le 16\mu.\end{align}
Indeed, this  implies that
$$
\fint_{B(0, 1)}|Du-DP|^2\,dx\le   16\lz|DP|
$$
as desired

To see \eqref{xe2.7}, by \eqref{xe2.5} and $\mu<1$, one has
\begin{align*}|Dv(x)-e_2|^2&= |Dv(x)|^2-2v_{x_2}+1\\
&\le  (1+2\mu)^2+1- 2v_{x_2} \\
&\le  2+4\mu+4\mu^2-2v_{x_2}\\
&\le 2(1+4\mu-u_{x_2}) .\end{align*}
We therefore obtain
 $$\int_{B(0,1)}|Dv(x)-e_2|^2\,dx \le
2\int_{B(0,1)} (1+4\mu- v_{x_2})\,dx=2\int_{|x_1|\le 1}\int_{-\sqrt{1-|x_1|^2}}^{\sqrt{1-|x_1|^2}}\big[1+4\mu -v_{x_2}\big]\,dx_2\,dx_1. $$
Note that
\begin{align*}&\int_{-\sqrt{1-|x_1|^2}}^{\sqrt{1-|x_1|^2}}\big[1+4\mu -v_{x_2}\big]\,dx_2\\
&\quad= 2(1+4\mu)\sqrt{1-|x_1|^2}-\big[v(x_1,\sqrt{1-|x_1|^2})-v(x_1,-\sqrt{1-|x_1|^2})\big]\\
&\quad=8\mu \sqrt{1-|x_1|^2}\\
&\quad\quad-  \Big\{[v (x_1,\sqrt{1-|x_1|^2})-Q(x_1,\sqrt{1-|x_1|^2})]-\big[v(x_1,-\sqrt{1-|x_1|^2})- Q(x_1,-\sqrt{1-|x_1|^2})\big] \Big\}\\
&\quad\le   8\mu.
\end{align*}
We therefore obtain  \eqref{xe2.7}.
 \end{proof}

We are ready to prove Proposition \ref{lem2.8}.
\begin{proof}[Proof of Proposition \ref{lem2.8}.]
Let $u$ be an $\fz$-harmonic function in a planar domain $\Omega$. Fix any $x_0\in\Omega$ and $r>0$ so that $
B(x_0,4r)\subset\Omega$. Let $P\in\cp$.
Without loss of generality we assume that $x_0=0$.
By  Lemma  \ref{lem2.7} we only need  to prove
\begin{equation}\label{xe2.8}\int_{B(0,r)}-\det D^2u\,dx\le C \inf_{P\in\cp}\fint_{B(0,2r)}|Du-DP|^2\,dx.
\end{equation}

Note that   $D^2 v=D^2(v-P) $, and hence
 $-\det D^2 v=-\det D^2(v-P) $
for any smooth function $v$.
 The distributional definition of $ -\det D^2 v$ then must coincide with that of $ - \det D^2(u-P) $, i.e.,
\begin{align*}
\int_{\Omega} (-\det D^2u) \phi^2\,dx
&= \frac12\int_{\Omega} |D(u-P)|^2 \Delta\phi^2-   D^2\phi^2 D(u-P)\cdot D(u-P) ]\,dx
\quad\forall \phi\in C^\fz_c(\Omega).
\end{align*}
Thus,
$$
 \int_{\Omega} (-\det D^2u) \phi^2\,dx\le  C\int_{\Omega} |Du-DP|^2  (|D^2\phi||\phi|+|D\phi|^2) \,dx  \quad\forall \phi\in C^\fz_c(\Omega).$$
Since $-\det D^2u\ge 0$ as proved \cite{kzz},
by a   choice of cut-off function $\phi$  as in \eqref{cut-off}, we have \eqref{xe2.8}.
 \end{proof}
To prove Theorem \ref{thm1.1} {$(ii)$},  we also need the following result by Crandall-Evans \cite{ce}.

\begin{lem}\label{lem2.9} If $u$ is $\fz$-harmonic   in $\rr^2$ and $\lip u(x_0)=\|Du\|_{L^\fz(\rr^2)}<\fz$ for some $x_0$,
then $u\in\cp$, i.e.,  $u(x)=b+a\cdot x$ in $\rr^2$ for some $b\in\rr$ and $a\in\rr^2$.
\end{lem}
This allows us to get the  following, via some argument  much similar to that for the linear approximations in \cite{ce}.
We give the details here for reader's convenience.
\begin{lem}\label{lem2.10} Let $u $ be an $\fz$-harmonic function in $\rr^2$ with $0<\|Du\|_{L^\fz(\rr^2)}<\fz$. Then
 there exists a
  subsequence $\{m_j\}_{j\in\nn}\subset\nn$  and  a
 vector $a\in\rr^2$ such that
\begin{equation}\label{xe2.10}
\lim_{j\to\fz}\sup_{B(0,4)}\left|\frac{u(m_j x)}{m_j}-a\cdot x\right|=0.
\end{equation}
\end{lem}
\begin{proof}
Without loss of generality, we may assume that $u(0)=0$.
Write $u_{m}(x)=u(mx)/m$ for $m\in\nn$.  Since $\|Du_m\|_{L^\fz(\rr^2)}= \|Du\|_{L^\fz(\rr^2)}<\fz$, we know that
  $\{u_m\}_{m\in\nn}$ is equicontinuous and locally uniformly bounded   in $\rr^2$. Thus,
there exists a subsequence $\{m_j\}\subset \nn$ such that
  $u_{m_j}$ converges locally uniformly to some continuous function $w\in C^0(\rr^2)$ with $w(0)=0$.
Moreover,    since  $\|Du_m\|_{L^\fz(\rr^2)}= \|Du\|_{L^\fz(\rr^2)}$, one has
$$
|w(x)-w(y)|=\lim_{k\to\fz}|u_{m_j}(x)-u_{m_j}(y)|\le  \|Du\|_{L^\fz(\rr^2)}|x-y|\quad\forall x,y\in \rr^2$$
and hence $\|Dw\|_{L^\fz(\rr^2)}\le  \|Du\|_{L^\fz(\rr^2)}<\fz$.
Due to the compactness of viscosity solutions,  $w$ is  an $\fz$-harmonic function in $\rr^2$.
To  see \eqref{xe2.10},  it suffices to prove  $\lip w(0)=\|Dw\|_{L^\fz(\rr^2)} $.
This allows us to apply Lemma \ref{lem2.9} to get that  $w(x)=a\cdot x  $ in $\rr^2$ for some $a\in\rr^2$.
Hence  $u_{m_j}$ converges locally uniformly to $a\cdot x  $, that is, \eqref{xe2.10} holds.

Finally, we show that  $\lip w(0)=\|Dw\|_{L^\fz(\rr^2)}$.  We always have
$\lip w(0)\le \|Dw\|_{L^\fz(\rr^2)}$. To see the converse,
recall that $w$ has the linear approximation property at  $0$,
that is, for any sequence $\{r_j\}_{j\in\nn}$ converging to $0$, {there are} a subsequence $\{r_{j_k}\}_{k\in\nn}$  and also a vector $e$ depending on    $\{r_{j_k}\}_{k\in\nn} $ such that
$$\lim_{k\to\fz}\sup_{z\in B(0,1)}\left|\frac{w( r_{j_k}z) }{r_{j_k}} - e\cdot z  \right|=0.$$
and $|e|=\lip w(0)$.
Therefore, for every $\lz>0$, there exists a $r_\lz\in(0,1)$ such that
$$
\sup_{B(0,2r_\lz)}\frac1{r_\lz}|w(x)-e\cdot x|\le  \lz.$$
On the other hand,  since $u_{m_j}\to w$ uniformly in $B(0,2)$, there  exists $j_\lz$ such that  for $j\ge j_\lz$,
$$|u_{m_j}(z)-w(z)|\le \lambda r_\lz\quad \forall \ z\in B(0,2).$$
Therefore,
$$
\sup_{B(0,2r_\lz )}\frac1{r_\lz}|  u_{m_j}( x)  -e\cdot x|\le 2 \lz,$$
or equivalently,
$$
\sup_{B(0,2m_j r_\lz )} \frac1{2m_j r_\lz}|  u ( x)  -e\cdot x|\le   \lz .$$
By \eqref{xe2.5}, one has
 $|Du(x)|\le |e|+4\lz$ for all $ x\in B(0,m_jr_\lz) $.
By sending $j\to\fz$,  we also have this inequality for all $ x\in\rr^2$.
By the arbitrariness of $\lz>0$   we have  $|e|\ge\|Du\|_{L^\fz(\rn)}  $
and hence $\lip w(0)\ge \|Du\|_{L^\fz(\rn)}$ as desired.
\end{proof}

\begin{proof}[Proof of Theorem \ref{thm1.1} {$(ii)$}]
Let $u $ be an $\fz$-harmonic function in $\rr^2$  with
$$\liminf_{R\to\fz}\inf_{c\in\rr}\frac1R\fint_{B(0,R)}|u(x)-c|\,dx<\fz.$$
For any $ c\in\rr$, since $u-c $ be an $\fz$-harmonic function in $\rr^2$, by {Theorem \ref{thm1.1} $(i)$} one has
$$\|Du\|_{L^\fz(B(0,R/2))}\le C \frac1R\fint_{B(0,R)}|u(x)-c|\,dx\quad\forall R>0.$$
Thus,
$$
\|Du\|_{L^\fz(B(0,R/2))}\le  C\inf_{c\in\rr}\frac1R\fint_{B(0,R)}|u(x)-c|\,dx\quad\forall R>0.
$$
This  implies
$$\|Du\|_{L^\fz(\rr^2)}=\liminf_{R\to\fz}\|Du\|_{L^\fz(B(0,R/2))}\le \liminf_{R\to\fz}\inf_{c\in\rr}\frac1R\fint_{B(0,R)}|u(x)-c|\,dx<\fz.$$

We assume that  $\|Du\|_{L^\fz(\rr^2)}>0$ { because} otherwise   $u$ is a constant.
By Lemma \ref{lem2.10},  there exists a
  subsequence $\{m_j\}_{j\in\nn}\subset\nn$  and  a
 vector $a\in\rr^2$ such that
\begin{equation}\label{xe2.12}
\lim_{j\to\fz}\sup_{B(0,4)}|\frac{u(m_j x)}{m_j}-a\cdot x|=0.
\end{equation}
From \eqref{xe2.12}, \eqref{xe1.1}, and Proposition \ref{lem2.8} we deduce
\begin{align*} \int_{\rr^2}|D|Du||^2\,dx&\le C\int_{\rr^2} -\det D^2u \,dx\\
&=C\lim_{j\to\fz}\int_{B(0,m_j)} -\det D^2u \,dx\\
&\le   C \lim_{j\to\fz}\sup_{x\in B(0,4 )} \Big|\frac{u(m_jx)}{m_j} -a\cdot x\Big|\\
&=0.
\end{align*}
Thus, $|Du|$ is a constant almost everywhere, and hence   $\|Du\|_{L^\fz(\rr^2)}=|Du(x_0)|$ for some $x_0\in\rr^2$.    By Lemma \ref{lem2.9},  we have $u\in\cp$ and hence $ u(x) =u(0)+a\cdot x$.
\end{proof}

\section{ A discussion of Definition \ref{def1.2}}
                    \label{sec3}
To see that Definition \ref{def1.2} makes sense, one must  prove the following Lemma \ref{lem3.1}.
For any $\bz>-1$, if  $v\in W^{1,1}_\loc(\Omega)$ and $|Dv|^{\bz}Dv\in W^{1,2}_\loc(\Omega)$,
then $-\det D[|Dv|^{\beta}Dv]$ is defined almost everywhere and $-\det D[|Dv|^{\beta}Dv]\in L^1_\loc(\Omega)$.

\begin{lem}\label{lem3.1}
For any $\bz>-1$, if  $v\in W^{1,1}_\loc(\Omega)$ and $|Dv|^{\bz}Dv\in W^{1,2}_\loc(\Omega)$,
then
\begin{align}\label{xe3.1}
\int_\Omega-\det D\big[|Dv|^{\beta}Dv\big] \psi\,dx
&=-\frac12\int_\Omega|Dv|^{2\beta} (D^2\psi Dv\cdot Dv) \,dx
+\frac{1}{2\beta+2}
\int_\Omega|Dv|^{2\beta+2} \bdz\psi \,dx\\
&\quad-\frac{\beta}{\beta+1}
\int_\Omega\big[D |Dv| ^{\bz+1 }\cdot D v\big] (Dv\cdot D\psi)|Dv|^{\beta-1} \,dx \ \ \forall \,\psi\in C^\fz_c(\Omega).\nonumber
\end{align}
\end{lem}

Recall that, in the case when   $\beta=0$, Lemma \ref{lem3.1} follows from \eqref{xe2.2} and a standard approximation via smooth functions.
Indeed, denote by $\{\eta_\ez\}_{\ez\in(0,1]}$ the standard smooth mollifier:
\begin{align}\label{moll}\mbox{
$\eta_\ez(x)=\ez^{-2}\eta(\ez^{-1}x)$, where  $\eta\in C^\fz(B(0,1))$ satisfying  $\eta\ge0$ and $\int_{B(0,1)}\eta\,dx=1$.}
\end{align}
Given any $v\in W^{2,2}_\loc(\Omega)$,  we know that \eqref{xe2.2} holds for $v\ast\eta_\ez$.
As $\ez\to0$, since $-\det D^2 (v\ast\eta_\ez)=-\det \big[(D^2v)\ast\eta_\ez\big]\to -\det D^2v$ in $L^1_\loc(\Omega)$,
$|D(v\ast\eta_\ez)|^2\to |Dv|^2$ in $L^1_\loc(\Omega)$ and $D(v\ast\eta_\ez) \otimes D (v\ast\eta_\ez)\to Dv\otimes Dv$ in $L^1_\loc(\Omega)$, we know that \eqref{xe2.2} holds for such $v$.

Below we prove  the  case $\beta\ne 0$ in Lemma \ref{lem3.1}.
We need several lemmas.
For any $v\in C^\fz(\Omega)$, in the case when $\bz\ge 0$, one has $|Dv|^\bz Dv\in W^{1,2}_\loc(\Omega)$,
 but when $-1<\bz<0$,   we do not necessarily have   $|Dv|^\bz Dv\in W^{1,2}_\loc(\Omega)$. Indeed,
 if $w(x)= x_1^2$ in $\rr^2$, then
 a direct calculation leads to
$$
|D[|Dw|^\bz Dw]|^2=  4(\bz+1)^2|x_1|^{2\bz} \quad\mbox{ when $x_1\ne 0$, }
$$
 which  does not belong to $ L^1_\loc(\rr^2)$  when $\bz\le -1/2$.
For this reason when $\bz<0$ we consider $-\det D[(|Dv|^2+\ez)^{{\bz/2}}Dv]$ with $\ez>0$.
We have the following result, whose proof is postponed to Section \ref{sec3.1}.

\begin{lem}\label{lem3.2}
Let $v\in C^\fz(\Omega)$. Given any $\bz\ge0$ and $\ez\ge 0$, or given any $\bz\in(-1,0)$ and $\ez>0$, one has
\begin{align}\label{xe3.2}
&\int_\Omega-\det D\big[(|Dv|^2+\ez)^{{\bz/2}}Dv\big] \psi \,dx\\
&\quad= -\frac12\int_\Omega(|Dv|^2+\ez)^{\bz } (D^2\psi Dv\cdot Dv) \,dx
+\frac{1}{2\beta+2}\int_\Omega(|Dv|^2+\ez)^{\beta+1} \bdz\psi\,dx\nonumber\\
&\quad\quad-\frac{\beta}{\beta+1}\int_\Omega(|Dv|^2+\ez)^{\frac{\beta-1}{2}}
\big[D(|Dv|^2+\ez)^{\frac{\bz+1}{2} }\cdot D v\big] (Dv\cdot D\psi) \,dx \quad\forall \psi\in C^\fz_c(\Omega) \nonumber.
\end{align}
\end{lem}

We also have the following two divergence structural formulae.
\begin{lem}\label{lem3.3}
Let $v\in C^\fz(\Omega)$. For any $\ez>0$ and $\bz>-1$, one has for any $\psi\in C^\fz_c(\Omega )$,
\begin{align*}
 \int_\Omega -\det D\big[( |Dv |^2+\ez) ^{{\bz/2}}Dv\big] \psi\,dx
&= \int_\Omega\Big\{
\big[ (|Dv  |^2+\ez) ^{{\bz /2}}v  _{x_2}\big]_{x_2}(|Dv  | ^2+\ez)^{{\bz/2 }}v _{x_1}\psi_{x_1} \\
&\quad\quad -
\big[ (|Dv  |^2+\ez) ^{{\bz/2 }}v  _{x_2}\big]_{x_1}
(|Dv  |^2+\ez) ^{{\bz/2 }}v  _{x_1}\psi_{x_2}\Big\} \,dx.
\end{align*}
\end{lem}

\begin{proof}
Write $F=(|Dv|^2+\ez)^{{\bz/2}}Dv$. By integration by parts, one has
\begin{align*}\int_\Omega-\det(DF )\psi\,dx
 &= -\int_\Omega  \left[( F _1)_{x_1}
(F _2)_{x_2}  -  (F _1)_{x_2}
(F _2)_{x_1}\right]\psi \,dx\\
&=\int_\Omega  \left[ F _1
(F _2)_{x_2}\psi_{x_1}\,dx -
F _1
( F _2)_{x_1}\psi_{x_2}\right]\,dx \end{align*}
as desired.
\end{proof}

\begin{lem}\label{lem3.4}
Let $v\in W^{1,1}_\loc(\Omega)$ satisfy $ |Dv|^\bz Dv\in W^{1,2}_\loc(\Omega)$ for some $ \bz>-1$. One has
\begin{align*}
&\int_\Omega -\det D\big[ |Dv| ^{{\bz}}Dv\big] \psi\,dx\\
& \quad= \int_\Omega\big[
( |Dv | ^{{\bz }}v _{x_2})_{x_2}|Dv | ^{{\bz }}v _{x_1}\psi_{x_1}   -
( |Dv | ^{{\bz }}v _{x_2})_{x_1}|Dv | ^{{\bz }}v _{x_1}\psi_{x_2}\big] \,dx
\quad\forall  \psi\in C^\fz_c(\Omega ).
\end{align*}
\end{lem}

\begin{proof}
For $\ez>0$, let $F^\ez=(F^\ez_1,  F^\ez_2)=(|Dv|^\bz Dv)\ast\eta_\ez
 \in C^\fz_\loc(\Omega)$, where $\eta_\ez$ is the standard smooth mollifier as in \eqref{moll}.
By $|Dv|^\beta Dv\in W^{1,2}_\loc(\Omega)$, we know
that $F^\ez \to |Dv|^\bz Dv$ in $W^{1,2}_\loc(\Omega)$ as $\ez \to 0$.  By this,
 integration by parts and
$(F^\ez_2)_{x_1x_2}=(F^\ez_2)_{x_2x_1}$,
 we have
\begin{align}
\int_\Omega-\det\big(D\big[|Dv|^\beta Dv\big]\big)\psi\,dx&=
\lim_{\ez\to 0}\int_\Omega-\det(DF^\ez)\psi\,dx\nonumber\\
&=-\lim_{\ez\to0}\int_\Omega  \big[( F^\ez_1)_{x_1}
(F^\ez_2)_{x_2}  -  (F^\ez_1)_{x_2}
(F^\ez_2)_{x_1}\big]\psi \,dx \nonumber\\
&=\lim_{\ez\to0}\int_\Omega  \big[ F^\ez_1
(F^\ez_2)_{x_2}\psi_{x_1}\,dx -
F^\ez_1
( F^\ez_2)_{x_1}\psi_{x_2}\big]\,dx \nonumber\\
&=\int_\Omega  \big[( |Dv | ^{{\bz }}v _{x_1})
( |Dv | ^{{\bz }}v _{x_2})_{x_2}\psi_{x_1}\,dx -  ( |Dv | ^{{\bz }}v _{x_1} )
( |Dv | ^{{\bz }}v _{x_2})_{x_1}\psi_{x_2}\big]\,dx\nonumber
\end{align}
for any $\psi\in C^\fz_c(\Omega )$. Hence we complete this proof.
\end{proof}

Moreover, we need the following approximation result.
Given any $0\ne \bz>-1$, let $v\in W^{1,1}_\loc(\Omega)$ satisfy $ |Dv|^\bz Dv\in W^{1,2}_\loc(\Omega)$.
Write  $g:={\rm div}(|Dv|^{\bz}Dv)  \in L^2_\loc (\Omega).$
Given  any $B=B(x_0,r)\Subset 4B\Subset\Omega$, one has   $g  \in L^2(4B ).$
For  any $\ez\in(0,r]$, set
$$g^{\ez}(x):=g\ast \eta_{\ez} (x)\quad \forall x\in 3B,$$
where $\{\eta_\ez\}_{\ez\in(0,1]}$ is the standard smooth mollifier as in \eqref{moll}.
Note that $g^\ez\in C^\fz(3B )$, $g^{\ez}\in L^2(3 B )$ uniformly in $\ez\in(0,r]$, and
$g^{\ez}\to g$ in $L^2(
3 B )$ as $\ez\to0$.
Since
$|Dv|^{\bz}Dv\in W^{1,2}(4B )$, by the Sobolev embedding theorem,
we have $|Dv|^{\bz}Dv\in L^q(4B )$,  and hence  $v\in W^{1,q}(4B )$,  for any $1<q<\fz$.
Consider the Dirichlet problem for the inhomogeneous $(2+\beta)$-Laplace  equation
\begin{align}\label{xe3.3}
{\rm div}\big((|Dw|^2+\ez)^{\frac{\bz}{2}}Dw\big)
=g^{\ez}\ {\rm in}\ 2B ;\ w=v\ {\rm on}\ \partial (2B).
\end{align}
There exists a unique smooth solution $v^{\ez}\in C^\fz(2B )\cap W^{1,2+\beta}(2B )\cap C^0(\overline {2B})$  to   \eqref{xe3.3}.
The following convergence result  plays a key role in the proof of Lemma \ref{lem3.1}.  Its proof  is postponed to  Section \ref{sec3.2}.

\begin{lem}\label{lem3.5} We have
\begin{itemize}
\item[$(i)$]   $Dv^\ez\to Dv$ in $L^{2+\bz}(2B )$ as $\ez\to 0$.
\item[$(ii)$]
$\big[|Dv^\ez|^2+\ez\big]^{\frac{\bz}2}Dv^\ez
\in W^{1,2}(2B )$ uniformly in $\ez\in(0,r]$;
$\big[|Dv^\ez|^2+\ez\big]^{\frac{\bz}2}Dv^\ez \to
|Dv|^{\bz}Dv$ in $L^q(B)$
for any $q\in(1,\fz)$ and weakly in $ W^{1,2}(B )$ as $\ez\to0$.
\item[$(iii)$]  $\big[|Dv^\ez|^2+\ez\big]^{\frac{\bz+1}2}\in W^{1,2}(B )$ uniformly in $\ez\in(0,r]$;
    $\big[|Dv^\ez|^2+\ez\big]^{\frac{\bz+1}2}\to
|Dv|^{\bz+1}$ in $L^q(B)$ for any $q\in(1,\fz)$ and weakly in $W^{1,2}(B )$  as $\ez\to0$.
\end{itemize}
\end{lem}

Now we are able to prove Lemma \ref{lem3.1} as below.

\begin{proof}[Proof of Lemma \ref{lem3.1}] Let $\bz>-1$ but $\beta\ne0$.
Up to a partition of unit, we only need to show that \eqref{xe1.3} for all $\psi\in  C_c^\fz(B)$ whenever  $B=B(x_0,r)$ with $4B \subset\Omega$.  Fix such a ball $B$.
Let $ v^\ez$ be as in Lemma \ref{lem3.5}.
Since
$\big[|Dv^\ez|^2+\ez\big]^{\frac{\bz}2}Dv^\ez \to
|Dv|^{\bz}Dv$ weakly in $W^{1,2}(B)$ as given in Lemma \ref{lem3.5} $(ii)$, by Lemmas \ref{lem3.4} and \ref{lem3.3}  we have
\begin{align*}
&\int_B -\det D[ |Dv| ^{{\bz}}Dv] \psi\,dx\\
& = \int_B\big[
( |Dv  | ^{{\bz }}v _{x_2})_{x_2}|Dv  | ^{{\bz }}v _{x_1}\psi_{x_1}   -
( |Dv  | ^{{\bz }}v _{x_2})_{x_1}|Dv | ^{{\bz }}v _{x_1}\psi_{x_2}\big] \,dx\\
&=\lim_{\ez\to0}\int_B\big[
( (|Dv^\ez |^2+\ez) ^{{\frac \bz 2}}v^\ez _{x_2})_{x_2}(|Dv^\ez | ^2+\ez)^{{\frac \bz 2 }}v^\ez _{x_1}\psi_{x_1}   -
( (|Dv^\ez |^2+\ez) ^{{\frac \bz 2 }}v^\ez _{x_2})_{x_1}(|Dv^\ez |^2+\ez) ^{{\frac \bz 2 }}v^\ez _{x_1}\psi_{x_2}\big] \,dx\\
&=\lim_{\ez\to0} \int_B -\det D\big[( |Dv^\ez|^2+\ez) ^{{\bz/2}}Dv^\ez\big] \psi\,dx
\quad\forall  \psi\in C^\fz_c(B).
\end{align*}
Note that by Lemma \ref{lem3.2},  \eqref{xe3.2} always holds  $v^\ez$ and $\psi\in C^\fz(B)$.
To get  \eqref{xe3.1} for $v$,  it then suffices to show that
\begin{align}
 \int_B(|Dv|^2+\ez)^{\bz } (D^2\psi Dv\cdot Dv) \,dx&\to \int_B |Dv| ^{2\bz } (D^2\psi Dv\cdot Dv) \,dx\label{xe3.4},\\
  \int_B(|Dv|^2+\ez)^{\beta+1} \bdz\psi\,dx&\to \int_B|Dv|^ {2\beta+2} \bdz\psi\,dx\label{xe3.5}, \ \mbox{and}\\
  \int_B(|Dv|^2+\ez)^{\frac{\beta-1}{2}}[D(|Dv|^2+\ez)^{\frac{\bz+1}{2} }\cdot D v] (Dv\cdot D\psi) \,dx&
\to
\int_B[D |Dv| ^{\bz+1 }\cdot D v] (Dv\cdot D\psi)|Dv|^{\beta-1} \,dx \label{xe3.6}
\end{align}
for any $\psi\in C_c^\fz(B)$.

By Lemma \ref{lem3.5} $(ii)$, we have $(|Dv^\ez|^2+\ez)^{\bz/2 }   Dv^\ez\to   |Dv|^{\bz}Dv $ in $L^2(B)$, which gives \eqref{xe3.5}.

To get \eqref{xe3.4}, we only need to show
\begin{equation}\label{xe3.7}
\mbox{$|(|Dv^\ez|^2+\ez)^{\bz }   Dv^\ez\otimes Dv^\ez\to |Dv|^{2\bz}Dv\otimes Dv$ in $L^2(B)$.} \end{equation}
Noting
$$|a\otimes a-b\otimes b|\le |a\otimes a-a\otimes b|+|a\otimes b-b\otimes b|\le |a-b|\big[|a|+|b|\big],$$
we have
\begin{align*}
&|(|Dv^\ez|^2+\ez)^{\bz }   Dv^\ez\otimes Dv^\ez- |Dv|^{2\bz}Dv\otimes Dv|\\
&\quad\le |(|Dv^\ez|^2+\ez)^{\bz/2 }   Dv^\ez-  |Dv|^{\bz}Dv|\big[(|Dv^\ez|^2+\ez)^{(\bz+1)/2 }    +|Dv|^{\bz+1}\big].
\end{align*}
Since  $(|Dv^\ez|^2+\ez)^{\bz/2 }   Dv^\ez\to   |Dv|^{\bz}Dv$   in $L^2(B)$,  $(|Dv^\ez|^2+\ez)^{(\bz+1)/2 } \in L^2(B)$ uniformly in $\ez>0$,
and $|Dv|^{\bz+1}\in L^2(B)$, we obtain  \eqref{xe3.7}.

Finally,
 \eqref{xe3.6} follows from $D(|Dv^\ez|^2+\ez)^{(\bz+1)/2 }    \to  D  |Dv|^{\bz+1}  $ weakly in $L^2(B )$  as  given in Lemma \ref{lem3.5} $(iii)$, and also
$$
\mbox{$|(|Dv^\ez|^2+\ez)^{(\bz-1)/2 }   Dv^\ez\otimes Dv^\ez\to |Dv|^{\bz-1}Dv\otimes Dv$ in $L^2(B)$,}$$
which is proved  similarly to   \eqref{xe3.7}.
\end{proof}

\subsection{Proof of Lemma \ref{lem3.2}}
                \label{sec3.1}
We first recall the following fundamental identity  \eqref{xe3.8}; see    \cite{kzz,dfzz}.
\begin{lem}\label{lem3.6}
 For any $v\in C^\fz(\Omega)$, we have
\begin{equation}\label{xe3.8}
   { |D^2vDv|^2} -  {\Delta v \Delta_\fz v } =\frac12
   \big[|D^2v|^2-(\Delta v)^2\big]|Dv|^2 \quad  \mbox{in $\Omega$}.
\end{equation}
\end{lem}

Next we build up the following  structural identity.
\begin{lem}\label{lem3.7}
For any   $v\in C^\fz(\Omega)$, $\bz\in\rr$, and $\ez>0$,  we have
\begin{align}\label{xe3.9}
&-\det D\big[(|Dv|^2+\ez)^{{\bz/2}}Dv\big]\\
&\quad =\frac{1}{2}(|Dv|^2+\ez)^{\bz }\big[|D^2v|^2-(\bdz v)^2\big]
+\bz (|Dv|^2+\ez)^{{\bz-1}}\big[|D^2v Dv|^2-\bdz v\bdz_\fz v\big]\quad{\rm in}\ \Omega. \nonumber
\end{align}
Moreover, if  in addition $|Dv|>0$ in $\Omega$, then \eqref{xe3.9} holds with $ \ez=0$.
\end{lem}

\begin{proof}
For $1\le i,j\le 2$, one has
 \begin{equation*}
   [(|Dv|^2+\ez)^{{\bz/2}}v_{x_i}]_{x_j}=(|Dv |^2+\ez)^{{\bz/2}}\Big[v_{x_ix_j}+\bz (|Dv |^2+\ez)^{-1}(\frac{|Dv|^2}2)_{x_j}v_{x_i}\Big].
 \end{equation*}
 Thus,
\begin{align*}
&\det D[ (|Dv|^2+\ez)^{{\bz/2}}Dv]  \\
&\quad=(|Dv |^2+\ez)^{\bz }
\Big[v_{x_1x_1}+\bz (|Dv |^2+\ez)^{-1}(\frac{|Dv |^2}2)_{x_1}v _{x_1}\Big]
\Big[v_{x_2x_2}+\bz (|Dv |^2+\ez)^{-1}(\frac{|Dv |^2}2)_{x_2}v _{x_2}\Big]\\
&\quad\quad -(|Dv |^2+\ez)^{\bz }\Big[v_{x_ix_j}+\bz (|Dv |^2+\ez)^{-1}(\frac{|Dv|^2}2)_{x_2}v_{x_1}\Big]\Big
[v _{x_2x_1}+\bz (|Dv |^2+\ez)^{-1}(\frac{|Dv |^2}2)_{x_1}v _{x_2}\Big]\\
&\quad= (|Dv |^2+\ez)^{\bz }[v _{x_1x_1}v _{x_2x_2}-v _{x_1x_2}v _{x_2x_1}] +\bz (|Dv |^2+\ez)^{\bz -1}\\
&\quad\quad\quad\times \Big\{[v _{x_1x_1}(\frac{|Dv |^2}2)_{x_2}v _{x_2}
+ v _{x_2x_2}(\frac{|Dv |^2}2)_{x_1}v _{x_1}]     - [v _{x_1x_2}(\frac{|Dv |^2}2)_{x_1}v _{x_2}+
v _{x_2x_1}(\frac{|Dv |^2}2)_{x_2}v _{x_1}]\Big\} \\
&\quad\quad +\bz^2(|Dv |^2+\ez)^{\bz -2}\Big[(\frac{|Dv |^2}2)_{x_1}v _{x_1}(\frac{|Dv |^2}2)_{x_2}v _{x_2}-(\frac{|Dv |^2}2)_{x_2}v _{x_1}(\frac{|Dv |^2}2)_{x_1}v _{x_2}\Big].
\end{align*}

Observe  that the last term is $0$, the first term equals $(|Dv |^2+\ez)^{\bz }\det D^2v$.
Regarding the second term, since
$$\Delta_\fz v=(\frac{|Dv |^2}2)_{x_1}v _{x_1}+(\frac{|Dv |^2}2)_{x_2}v _{x_2}$$
and
\begin{align*}
|D^2v  Dv |^2=
& v _{x_1x_1}(\frac{|Dv |^2}2)_{x_1}v _{x_1}+v _{x_2x_2}(\frac{|Dv |^2}2)_{x_2}v _{x_2}
+v _{x_1x_2}(\frac{|Dv |^2}2)_{x_1}v _{x_2}+ v _{x_2x_1}(\frac{|Dv |^2}2)_{x_2}v _{x_1},\end{align*}
we have
\begin{align*}& \left[v _{x_1x_1}(\frac{|Dv |^2}2)_{x_2}v _{x_2}
+ v _{x_2x_2}(\frac{|Dv |^2}2)_{x_1}v _{x_1}\right]
-\left[v _{x_1x_2}(\frac{|Dv |^2}2)_{x_1}v _{x_2}+
v _{x_2x_1}(\frac{|Dv |^2}2)_{x_2}v _{x_1}\right]\\
&\quad = v _{x_1x_1}\Delta_\fz v +v _{x_2x_2}\Delta_\fz v-  v _{x_1x_1}(\frac{|Dv |^2}2)_{x_1}v _{x_1}-v _{x_2x_2}(\frac{|Dv |^2}2)_{x_2}v _{x_2}\\
&\quad\quad- \left[v _{x_1x_2}(\frac{|Dv |^2}2)_{x_1}v _{x_2}+
v _{x_2x_1}(\frac{|Dv |^2}2)_{x_2}v _{x_1}\right]\\
&\quad=\bdz v \bdz_\fz v-
|D^2v  Dv |^2.
\end{align*}
Thus,   the second term equals $-\bz (|Dv|^2+\ez)^{{\bz-1}}\big[|D^2v Dv|^2-\bdz v\bdz_\fz v\big]$. We therefore obtain \eqref{xe3.9}.
Finally we note that if $|Dv(x)|>0$, the above argument holds with $\ez=0$.
 This completes the proof of Lemma \ref{lem3.7}.
\end{proof}

We are ready to prove Lemma \ref{lem3.2}.

\begin{proof}[Proof of Lemma \ref{lem3.2}]
 By \eqref{xe2.1} and  integration by parts, for any $\psi\in C^\fz_c(\Omega)$   we have
\begin{align*}
& \int_\Omega(|Dv|^2+\ez)^{\bz }\big[|D^2v|^2-(\bdz v)^2\big]\psi\,dx\\
&\quad=  \int_\Omega(|Dv|^2+\ez)^{\bz }{\rm div}( D^2v Dv-\bdz v Dv)\psi\,dx\\
&\quad= -2\bz \int_\Omega(|Dv|^2+\ez)^{{\bz-1}}\big[|D^2v Dv|^2-\bdz_\fz v\bdz v\big]\psi\,dx\\
&\quad\quad+ \int_\Omega(|Dv|^2+\ez)^{\bz }[\bdz v Dv\cdot D\psi-D^2v Dv\cdot D\psi] \,dx.
\end{align*}
From this and  \eqref{xe3.9}, it follows that
\begin{align*}
\int_\Omega-\det D\big[(|Dv|^2+\ez)^{{\bz/2}}Dv\big] \psi\,dx
&=\frac12\int_\Omega(|Dv|^2+\ez)^{\bz }\big[\bdz v Dv\cdot D\psi-D^2v Dv\cdot D\psi\big] \,dx.
\end{align*}
By integration by parts again, we have
\begin{align*}
&\frac12\int_\Omega\bdz v(|Dv|^2+\ez)^{\bz }( Dv\cdot D\psi) \,dx\\
&\quad=
-\frac12\int_\Omega(|Dv|^2+\ez)^{\bz } (D^2\psi Dv\cdot Dv) \,dx\\
&\quad\quad-\frac12\int_\Omega(|Dv|^2+\ez)^{\bz } (D^2v Dv\cdot D\psi)  \,dx
-\frac12\int_\Omega\big[D(|Dv|^2+\ez)^{\bz }\cdot D v\big] (Dv\cdot D\psi) \,dx.
\end{align*}
Noting
$$(|Dv|^2+\ez)^{\bz }D^2v Dv=\frac{D(|Dv|^2+\ez)^{\beta+1}}{2\beta+2},$$
one further gets
\begin{align*}
-\int_\Omega (|Dv|^2+\ez)^{\bz } D^2v Dv\cdot D\psi \,dx
&=-\frac{1}{2\beta+2}\int_\Omega\big[D(|Dv|^2+\ez)^{\beta+1}\cdot D\psi\big]  \,dx\\
&=\frac{1}{2\beta+2}\int_\Omega(|Dv|^2+\ez)^{\beta+1}  \bdz\psi \,dx.
\end{align*}

We also observe that
$$D(|Dv|^2+\ez)^{\beta}=\frac{2\beta}{\beta+1}
(|Dv|^2+\ez)^{\frac{\beta-1}{2}}D(|Dv|^2+\ez)^{\frac{\beta+1}{2}}.$$
Thus,
\begin{align*}
&-\frac12\int_\Omega\big[D(|Dv|^2+\ez)^{\bz }\cdot D v\big] (Dv\cdot D\psi) \,dx\\
&\quad=-\frac{\beta}{\beta+1}\int_\Omega(|Dv|^2+\ez)^{\frac{\beta-1}{2}}
\big[D(|Dv|^2+\ez)^{\frac{\bz+1}{2} }\cdot D v\big] (Dv\cdot D\psi) \,dx.
\end{align*}
Combining all above together we obtain the desired identity \eqref{xe3.2}.
\end{proof}

\subsection{Proof of Lemma \ref{lem3.5}}
                            \label{sec3.2}
\begin{proof}[Proof of Lemma  \ref{lem3.5}.]
Up to some scaling and translation, we assume that $x_0=0$ and $r=1$. Write $B_m=B(0,m)= m B(0,1)$ for $m\ge 1$.
We divide the proof into three steps.

\medskip
{\bf Step 1}:  Prove   $v^\ez\in L^2(B_2)$ and    $Dv^\ez\in L^{ 2+\bz}(B_2) $ uniformly in $\ez\in(0,1]$.

Since $v^\ez-v\in W^{1,\bz+2}_0(B_2)$, by the Sobolev-Poincar\'e inequality, it suffices to prove that $Dv^\ez\in L^{2+\bz}(B_2)$ uniformly in $\ez\in(0,1]$.
Choosing the test function $v^\ez-v\in W^{1,\bz+2}_0(B_2)$ to equation \eqref{xe3.3} we get
\begin{align}\label{xe3.11}
\int_{B_2}(|Dv^\ez|^2+\ez)^{\frac \bz 2}Dv^\ez\cdot (Dv^\ez-Dv)\,dx=
-\int_{B_2}g^\ez(v^\ez-v)\,dx,
\end{align}
or equivalently,
$$\int_{B_2}(|Dv^\ez|^2+\ez)^{\frac \bz 2}|Dv^\ez|^2 \,dx=
\int_{B_2}(|Dv^\ez|^2+\ez)^{\frac \bz 2}Dv^\ez\cdot  Dv\,dx
-\int_{B_2}g^\ez(v^\ez-v)\,dx.$$

Young's inequality yields that
\begin{align*}
\int_{B_2}(|Dv^\ez|^2+\ez)^{\frac \bz 2}Dv^\ez\cdot  Dv\,dx
&\le \int_{B_2}(|Dv^\ez|^2+\ez)^{\frac {\bz+1} 2}| Dv|\,dx\\
&\le  \frac1{2^{4+\bz}}\int_{B_2}(|Dv^\ez|^2+\ez)^{\frac \bz 2+1}\,dx+C(\beta)\int_{B_2}|Dv|^{2+\bz}\,dx\\
&\le  \frac14\int_{B_2} |Dv^\ez|  ^{  \bz+2 }\,dx+\frac14  +C(\beta)\int_{B_2}|Dv|^{2+\bz}\,dx.
\end{align*}

By H\"older's inequality, the Sobolev-Poincar\'e inequality, and Young's inequality, one has
\begin{align*}
-\int_{B_2}g^\ez(v^\ez-v)\,dx&\le \Big(\int_{B_2}(g^\ez)^2\,dx\Big)^{1/2} \Big(\int_{B_2}|v^\ez-v|^2\,dx\Big)^{1/2}\\
&\le C\Big(\int_{B_{3}}g^2\,dx\Big)^{1/2} \Big(\int_{B_2}|Dv^\ez-Dv|^{2+\beta}\,dx\Big)^{\frac 1{ 2+\beta}} \\
&\le   C\Big(\int_{B_{3}}g^2\,dx\Big)^{\frac{2+\beta}{2(1+\beta)}} + \frac14 \int_{B_2}|Dv^\ez  |^{2+ \bz}\,dx+   C \int_{B_2}|Dv|^{2+\bz} \,dx.
\end{align*}
Therefore,
 we obtain
\begin{align*}
\int_{B_2}|Dv^\ez|^{\bz+2}\,dx
&\le C( \bz)\int_{B_2}|Dv|^{\bz+2}\,dx +\Big(\int_{B_{3}}g^2\,dx\Big)^{\frac{2+\beta}{2(1+\beta)}}+C. \end{align*}
Thus,
$Dv^\ez\in L^{\bz+2}(B_2)$ uniformly in $\ez\in (0,1]$.

\medskip
{\bf Step 2}: Prove
$v^\ez\to v$ in $W^{1,2+\bz}(B_2)$ as $\ez\to0$.

 Since $g={\rm div}(|Dv|^\bz Dv)$  and $v^\ez-v\in W^{1,p}_0(B_2)$ we   have
 $$ -\int_{B_2} g   (v^\ez-v)\,dx=\int_{B_2} |Dv|^{\bz}Dv \cdot (Dv^\ez-Dv)\,dx.$$
By this and \eqref{xe3.11},   one has
\begin{align*}
&\int_{B_2} (|Dv^\ez|^2+\ez)^{\frac \bz2}Dv^\ez \cdot (Dv^\ez-Dv)\,dx\nonumber \\
&\quad=\int_{B_2} |Dv|^{\bz}Dv \cdot (Dv^\ez-Dv)\,dx+\int_{B_2}(g^\ez-g)(v^\ez-v)\,dx,
\end{align*}
and hence
\begin{align*}
&\int_{B_2}\left((|Dv^\ez|^2+\ez)^{\frac \bz2}Dv^\ez-(|Dv|^{2}+\ez)^{\frac \bz2}Dv\right)\cdot (Dv^\ez-Dv)\,dx\nonumber \\
&\quad=\int_{B_2}\left(|Dv|^{\bz}Dv-(|Dv|^{2}+\ez)^{\frac \bz2}Dv\right)\cdot (Dv^\ez-Dv)\,dx+\int_{B_2}(g^\ez-g)(v^\ez-v)\,dx.
\end{align*}

Observe that, when $\bz>-1$, it holds that
\begin{align*}
(|\xi|^2+|\eta|^2+\ez)^{\frac \bz2}|\xi-\eta |^2\le C(\bz)\left((|\xi|^{2}+\ez)^{\frac \bz2}\xi-(|\eta|^{2}+\ez)^{\frac \bz2}\eta\right)\cdot (\xi-\eta)\quad\forall \xi,\eta\in\rr^2.
\end{align*}
Thus,
\begin{align*}
&\int_{B_2}(|Dv^\ez|^2+|Dv|^2+\ez)^{\frac \bz2}|Dv^\ez-Dv|^2\,dx\nonumber\\
&\quad\le C(\bz)\int_{B_2}\left(|Dv|^{\bz}Dv-(|Dv|^{2}+\ez)^{\frac \bz2}Dv\right)\cdot (Dv^\ez-Dv)\,dx+\int_{B_2}(g^\ez-g)(v^\ez-v)\,dx.
\end{align*}
By H\"older's inequality, we have
\begin{align} \label{xe3.12}
&\int_{B_2}(|Dv^\ez|^2+|Dv|^2+\ez)^{\frac \bz2}|Dv^\ez-Dv|^2\,dx \\
&\quad\le C(\bz)\left(\int_{B_2}\left||Dv|^{\bz}Dv-(|Dv|^{2}+\ez)^{\frac \bz2}Dv
\right|^{\frac{\bz+2}{\bz+1}}\,dx\right)^{\frac{\bz+1}{\bz+2}}\left(
\int_{B_2}\big[|Dv^\ez|^{\bz+2}+|Dv|^{\bz+2}\big]\right)^{\frac{1}{\bz+2}}\nonumber\\
&\quad\quad+C(\bz)\left(\int_{B_2}|g^\ez-g|^2\,dx\right)^{\frac 12}\left(
\int_{B_2}\big[| v^\ez|^2+| v|^2\big]\,dx\right)^{\frac12}.\nonumber
\end{align}
Recalling  that $v^\ez\in W^{1,\bz+2}(B_2)$ uniformly in $\ez\in(0,1]$ as given in the step 1, and noting that
 $ -(|Dv|^{2}+\ez)^{\frac \bz2}Dv \to |Dv|^{\bz}Dv$ in $L^{\frac{\bz+2}{\bz+1}}(B_2)$ as $\ez\to0$,
we deduce that the first term in the right-hand side of \eqref{xe3.12}  tends to zero as $\ez\to 0$.
Since $v^\ez\in    W^{1,\beta+2}(B_2) $ uniformly in $\ez\in(0,1]$ and recall $g^\ez\to g$ in $L^2(B_2)$,  the second term of right-hand side of \eqref{xe3.12} tends to $0$ as $\ez\to0$.  Thus,
\begin{align}\label{xe3.13}
&\int_{B_2}(|Dv^\ez|^2+|Dv|^2+\ez)^{\frac \bz2}|Dv^\ez-Dv|^2\,dx
\to 0.
\end{align}

If $\bz>0$,  since
 $$
 |Dv^\ez-Dv|^{2+\bz}\le C(\bz)(|Dv^\ez|^2+|Dv|^2+\ez)^{\frac \bz2}|Dv^\ez-Dv|^2,
 $$
we have $Dv^\ez\to Dv$ in $L^{2+\bz}(B_2)$ as $\ez\to 0$.
 Thus, by $v^\ez-v\in W^{1,2+\bz}_0(B_2)$ and using the
Sobolev inequality we always have $v^\ez\to v$ in $W^{1,2+\bz}(B_2)$ as $\ez\to0$.

If $\bz\in (-1,0)$, by H\"older's inequality we get
\begin{align*}
&\int_{B_2}|Dv^\ez-Dv|^{\bz+2}\,dx\\
&\quad\le \left(
\int_{B_2}(|Dv^\ez|^2+|Dv|^2+\ez)^{\frac \bz2}|Dv^\ez-Dv|^2\,dx\right)^{\frac{\bz+2}{2}}  \left(
\int_{B_2}(|Dv^\ez|^2+|Dv|^2+\ez)^{\frac {\bz+2}2}\,dx\right)^{\frac{-\bz}{2}}.
\end{align*}
By Step 1, $v^\ez\in W^{1,2+\bz}(B_2)\cap L^2(B_2)$ uniformly in
$\ez\in(0,1)$.  Then \eqref{xe3.13}  yields that $Dv^\ez\to Dv$ in $L^{2+\bz}(B_2)$ as $\ez\to 0$.
We further conclude $v^\ez\to v$ in $W^{1,2+\bz}(B_2)$ as $\ez\to0$.

\medskip
{\bf Step 3}: Prove $(|Dv^\ez|^2+\ez)^{\frac \bz2}Dv^\ez\in W^{1,2}(B_1)$ uniformly
in $\ez \in(0,1]$ and $(|Dv^\ez|^2+\ez)^{\frac \bz2}Dv^\ez \to |Dv|^{\bz}Dv$ weakly
in $W^{1,2}(B_1)$ as $\ez\to0$.

By the local second order estimates   in \cite[Theorem 2.2]{cm} , we have
\begin{align*}
\int_{B_1}\left|D\big[(|Dv^\ez|^2+\ez)^{\frac \bz2}Dv^\ez\big]\right|^2\,dx\le C_0\int_{B_2}(g^\ez)^2\,dx
+C_0\left(\int_{B_2}\big[|Dv^\ez|^{\bz+1}+\ez^{\frac{\bz+1}{2}}\big]\,dx\right)^2.
\end{align*}
Sine $g^\ez\in L^2(B_2)$ uniformly in $\ez\in(0,1]$ and, by the step 1,
$Dv^\ez\in L^{2+\bz}(B_2)$ uniformly in $\ez\in(0,1]$, one has
$D\big[(|Dv^\ez|^2+\ez)^{\frac \bz2}Dv^\ez\big]\in L^2(B_1)$ uniformly
in $\ez\in(0,1]$. By the Sobolev-Poincar\'e inequality, we also have
\begin{align*}
\int_{B_1}|(|Dv^\ez|^2+\ez)^{\frac \bz2}Dv^\ez|^{2}\,dx
&\le\int_{B_1}\left|(|Dv^\ez|^2+\ez)^{\frac \bz2}Dv^\ez-\fint_{B_1}(|Dv^\ez|^2+\ez)^{\frac \bz2}Dv^\ez\,dx\right|^2\,dx\\
&\quad+\left|\int_{B_1}(|Dv^\ez|^2+\ez)^{\frac \bz2}Dv^\ez\,dx\right|^2\\
&\le C_0\int_{B_1}|D[(|Dv^\ez|^2+\ez)^{\frac \bz2}Dv^\ez]|^2\,dx
+C_0\left|\int_{B_1}[|Dv^\ez|^{\bz+1}+\ez^{\frac{\bz+1}{2}}]\,dx\right|^2.
\end{align*}
Thus $(|Dv^\ez|^2+\ez)^{\frac \bz2}Dv^\ez\in L^{2}(B_1)$ uniformly
in $\ez\in(0,1]$.

 Therefore, by weak compactness of Sobolev space there exists $f\in W^{1,2}(B_1)$ such
 that {along a subsequence}
 $$D\big[(|Dv^\ez|^2+\ez)^{\frac \bz2}Dv^\ez\big]\to Df\quad {\rm weakly\ in}\ L^2(B_1)$$
and
$$
(|Dv^\ez|^2+\ez)^{\frac \bz2}Dv^\ez\to f\quad{\rm in}\quad L^2(B_1).
$$
By Step 2,   $Dv^\ez\to Dv$  in $L^{2+\bz}(B_2)$ as $\ez\to0$,
and hence
$(|Dv^\ez|^2+\ez)^{\frac \bz2}Dv^\ez\to |Dv|^\bz Dv$ almost everywhere along a subsequence.
We conclude that $f=|Dv|^{\bz}Dv$ in $B_1$ as desired.

\medskip
{\bf Step 4}: Prove $(|Dv^\ez|^2+\ez)^{\frac {\bz+1}2}\in W^{1,2}(B_1)$ uniformly
in $\ez\in(0,1]$ and $(|Dv^\ez|^2+\ez)^{\frac {\bz+1}2} \to |Dv|^{\bz+1}$ weakly
in $W^{1,2}(B_1)$ as $\ez\to0$.

Note that
$$ |D (|Dv^\ez|^2+\ez)^{\frac {\bz+1}2}|^2=(\bz+1)^2|(|Dv^\ez|^2+\ez)^{  {\bz-1}} | D^2v^\ez Dv^\ez|^2     $$
and
\begin{align*}
\Big|D\big[(|Dv^\ez|^2+\ez)^{\frac \bz2}Dv^\ez\big]\Big|^2&= (|Dv^\ez|^2+\ez)^{ \bz }|D^2v^\ez|^2+ 2\bz
|(|Dv^\ez|^2+\ez)^{  {\bz-1}}  |D^2v^\ez Dv^\ez|^2\\
&\quad+\bz^2(|Dv^\ez|^2+\ez)^{  {\bz-2}}  |Dv^\ez|^2 |D^2v^\ez Dv^\ez|^2.
\end{align*}
If $\bz>0$, then
$$
|D (|Dv^\ez|^2+\ez)^{\frac {\bz+1}2}|^2\le (\bz+1)^2(|Dv^\ez|^2+\ez)^{ \bz }|D^2v^\ez|^2 \le C(\beta) |D[(|Dv^\ez|^2+\ez)^{\frac \bz2}Dv^\ez]|^2.
$$
If $\bz\in(0,1)$,
\begin{align*}
\Big|D\big[(|Dv^\ez|^2+\ez)^{\frac \bz2}Dv^\ez\big]\Big|^2
&= (|Dv^\ez|^2+\ez)^{{\bz-1}} \big[|D^2v^\ez|^2|Dv^\ez|^2+  2\bz|D^2v^\ez Dv^\ez|^2+ \bz^2   |D^2v^\ez Dv^\ez|^2\big]\\
&\quad +
(|Dv^\ez|^2+\ez)^{\bz-2}
\big[(|Dv^\ez|^2+\ez)\ez|D^2v^\ez|^2- \bz^2\ez |D^2v^\ez Dv^\ez|^2\big]\\
&\ge (\bz+1)^2 (|Dv^\ez|^2+\ez)^{  {\bz-1}}   |D^2v^\ez Dv^\ez|^2\\
&=  |D (|Dv^\ez|^2+\ez)^{\frac {\bz+1}2}|^2.
\end{align*}
Thus by Step 3, $(|Dv^\ez|^2+\ez)^{\frac {\bz+1}2}\in W^{1,2}(B_1)$ uniformly
in $\ez\in(0,1]$, and also $(|Dv^\ez|^2+\ez)^{\frac {\bz+1}2} \to |Dv|^{\bz+1}$ weakly
in $W^{1,2}(B_1)$ as $\ez\to0$.
\end{proof}

\section{Some properties of  distributional Jacobian determinant}
                        \label{sec4}

We  build up the following stability result.
\begin{lem} \label{lem4.1}
Let $\bz>-1$. If $v_j\to v$ in $W^{1,2+\bz}_\loc(\Omega)$ as $j\to\fz$ and $\bz |Dv_j|^{\bz+1}\in W^{1,2}_\loc(\Omega)$ uniformly in $j$, then
\begin{enumerate}
\item[$(i)$]
$\bz |Dv_j|^{\bz+1}\to\bz |Dv|^{\bz+1}$ in $L^q_\loc(\Omega)$ for any $q>1$ and weakly in $W^{1,2}_\loc(\Omega)$;

\item[$(ii)$]
$\bz |Dv_j|^{\bz }Dv_j\to\bz |Dv|^{\bz }Dv$  and $\bz |Dv_j|^{\bz-1 }Dv_j\otimes Dv_j \to
\bz |Dv|^{\bz-1 }Dv\otimes Dv$  in $L^q_\loc(\Omega)$ for any $q>1$.

\item[$(iii)$]   $-\det D\big[|Dv_j|^\bz Dv_j\big]\to -\det D\big[|Dv|^\bz Dv\big]$ in the distributional sense,
i.e.,
\begin{align}\label{xe4.1}\int_\Omega-\det D[|Dv_j|^\bz v_j]\psi\,dx\to \int_\Omega -\det D[|Dv|^\bz Dv] \psi\,dx\quad\forall\psi\in C_c^\fz(\Omega).
\end{align}
\end{enumerate}
\end{lem}

\begin{proof} Case $\bz=0$ is easy. We only consider the case $\bz\ne0$.
Since $  |Dv_j|^{\bz+1}\in W^{1,2}_\loc(\Omega)$,  by the {compact embedding theorem},  there is a function $f\in W^{1,2}_\loc(\Omega)$ such that one has  that
  $ |Dv_j|^{\bz+1}\to {f}$  in $L^q_\loc(\Omega)$ for any $1<q<\fz$ and  weakly in $W^{1,2}_\loc(\Omega)$ as $j\to\fz$ up to some subsequence.
 Since $ Dv_j \to  Dv  $ in $L^1_\loc(\Omega)$, we know that   $f=|Dv|^{\bz+1}$.
Thus $|Dv|^{\bz+1}\in W^{1,2}_\loc(\Omega)$
and
 $|Dv_j|^{\bz+1}$ converges {strongly} to $|Dv|^{\bz+1}$ in $L^q_\loc(\Omega)$ for any $1<q<\fz$ and  weakly in $W^{1,2}_\loc(\Omega)$. Therefore, $(i)$ holds.

To see $(ii)$,    observe that if $Dv(x)=0$,  one has
$$|
|Dv_j|^{\bz } Dv_j - |Dv |^{\bz } Dv  |\le ||Dv_j|^{\bz+1} -|Dv|^{\bz+1}|,$$
and
$$|
|Dv_j|^{\bz-1} Dv_j\otimes Dv_j- |Dv |^{\bz-1} Dv \otimes Dv|\le ||Dv_j|^{\bz+1} -|Dv|^{\bz+1}|.$$
 If $Dv(x)\ne 0$,  one has
$$|
|Dv_j|^{\bz } Dv_j - |Dv |^{\bz } Dv  |\le ||Dv_j|^{\bz+1}  -|Dv |^{\bz+1}|+|Dv |^{\bz+1} |\overline{Dv_j}  -\overline{ Dv }|,$$
and
$$|
|Dv_j|^{\bz-1} Dv_j\otimes Dv_j- |Dv |^{\bz-1} Dv \otimes Dv|\le
 ||Dv_j|^{\bz+1}  -|Dv |^{\bz+1}|+ |Dv |^{\bz+1}| \overline{Dv_j}\otimes \overline{Dv_j} -\overline{Dv }   \otimes \overline{Dv}|$$
where we set  $\overline{\xi}=\xi/|\xi|$ when $\xi\ne 0$ and $\overline{\xi}=0$ when  $\xi=0$.
Since along a subsequence $Dv_j\to Dv$ almost everywhere, we know that
$\overline{Dv_j}\otimes \overline{Dv_j}\to \overline{Dv}\otimes \overline{Dv}$ almost everywhere in
$\Omega\setminus\{x,Dv(x)=0\}$.
By the Lebesgue  dominated convergence, we conclude $(ii)$ from above and $(i)$.

 Finally, note that  $(i)$ and $(ii)$ implies
\begin{align*}
&
\int_\Omega|Dv_j|^{2\beta+2} \bdz\psi \,dx \to
\int_\Omega|Dv |^{2\beta+2} \bdz\psi \,dx, \\
&  \int_\Omega|Dv_j |^{2\beta} (D^2\psi Dv_j \cdot Dv_j ) \,dx \to  \int_\Omega|Dv |^{2\beta} (D^2\psi Dv \cdot Dv ) \,dx, \ \mbox{and}\\
&
\int_\Omega\big[D |Dv_j | ^{\bz+1 }\cdot D  v_j \big]
(Dv_j\cdot D\psi)|Dv_j |^{\beta-1} \,dx\to
\int_\Omega\big[D |Dv  | ^{\bz+1 }\cdot D  u  \big] (Dv^\ez \cdot D\psi)|Dv  |^{\beta-1} \,dx
\end{align*}
for all $\psi\in C_c^\fz(\Omega)$.
We conclude \eqref{xe4.1} from  these and the definitions of  $-\det D[|Dv_j|^\bz Dv_j]$ and
$  -\det D[|Dv|^\bz Dv]$.
\end{proof}

\begin{lem} \label{lem4.2}

Let $ v\in C^\fz(\Omega)$ satisfy $\bz|Dv|^{\bz+1}\in W^{1,2}_\loc(\Omega)$ for some $\bz>-1$.
We have
\begin{align*}
\lim_{\ez\to0}\int_\Omega-\det D\big[(|Du |^2+\ez)^{\bz/2} D u \big]\psi\,dx
\to \int_\Omega -\det D\big[|Du|^\bz Du\big] \psi\,dx\quad \forall\ \psi\in C_c^\fz(\Omega).
\end{align*}
\end{lem}

\begin{proof} The case $\bz=0$ is easy. The case $\bz\ne0$ would follow  if we could let  $\ez\to 0$ in \eqref{xe3.2} by Definition 1.3. To this end, it suffices to  { build} up the following convergence.

Firstly, since  $(|Dv|^2+1)^{\bz+1}\in L^q_\loc(\Omega)$ for $1<q<\fz$, by the Lebesgue  dominated convergence one has
$$\mbox{$(|Dv|^2+\ez)^{{\bz }}|Dv|^{ 2}\to |Dv| ^{{2\bz +2}}$  and $ (|Dv|^2+\ez)^{\bz}Dv\otimes Dv \to   |Dv| ^{2\bz}Dv\otimes Dv $
in $L^q_\loc(\Omega)$}.
$$
Similarly,
$$\mbox{$ (|Dv|^2+\ez)^{(\bz-1)/2}Dv\otimes Dv \to   |Dv| ^{(\bz-1)/2}Dv\otimes Dv $
in $L^q_\loc(\Omega)$}.
$$
Moreover, we observe that
\begin{equation}
                        \label{eq10.41}
 D(|Dv|^2+\ez)^{(\bz+1)/2}\to D|Dv|^{\bz+1}\text{ weakly in } L^2_\loc(\Omega).
\end{equation}
Indeed, when $\bz>0$,  since $v\in C^2(\Omega)$ one has
$$
D(|Dv|^2+\ez)^{(\bz+1)/2}= C(\bz)(|Dv|^2+\ez)^{( \bz-1)/2} D^2v Dv\in L^2_\loc(\Omega) \mbox{ uniformly in $\ez\in(0,1)$}
$$
and hence
 $D(|Dv|^2+\ez)^{(\bz+1)/2} \to D|Dv|^{\bz+1}$ { weakly} in $L^2_\loc(\Omega)$.
When $-1<\bz<0$, by the assumption $D|Du|^{\bz+1}\in L^2_\loc(\Omega)$ we have
$$
D(|Dv|^2+\ez)^{(\bz+1)/2}= C(\bz)(|Dv|^2+\ez)^{(\bz-1)/2}|Dv|^{1-\bz} D|Dv|^{\bz+1}\in L^2_\loc(\Omega) \mbox{ uniformly in $\ez\in(0,1)$}.
$$
Thus, together with  $(|Dv|^2+\ez)^{(\bz+1)/2}\to |Dv|^{\bz+1}$ weakly in $L^2_\loc(\Omega)$, we conclude \eqref{eq10.41} as desired.
\end{proof}

\section{Proof of Theorem \ref{thm1.3}}
            \label{sec5}

\begin{proof}[Proof of Theorem \ref{thm1.3}]
Let $u$ be any $\fz$-harmonic function   in planar domain $\Omega$. Since $u\in C^{0,1}_{{ \text{loc}}} (\Omega)$ and $|Du|^{\bz+1}\in W^{1,2}_\loc(\Omega)$,  by  Definition \ref{def1.2}, the distributional Jacobi $ -\det D[|Du  |^\beta Du ]$  is {well} defined. We proceed as below to show that $ -\det D[|Du  |^\beta Du ]\in \cm(\Omega)$
with the lower bound \eqref{xe1.7} and the upper bound \eqref{xe1.8}.

\medskip

 {\bf Step 1.} Given any smooth subdomain
$U\Subset\Omega$, for any $\ez\in(0,1)$ denote by  $u^\ez\in C^\fz(U)\cap C^0(\overline U)$ the
unique solution to the equation
$$\mbox{ ${\rm div}(e^{\frac1{2\ez} |Du^\ez|^2}Du^\ez)= \frac 1 \ez e^{\frac1{2\ez} |Du^\ez|^2}(\Delta_\fz u^\ez+\ez\Delta u^\ez)=0$ with $u^\ez=u$ on $\partial U$. }  $$
It was shown in \cite{e03}
 $$
\limsup_{\ez\to0}\|Du^\ez\|_{L^\fz(V)}\le \|Du\|_{L^\fz(U)}\quad\forall V\Subset U.$$
By \cite{kzz}, for any $\bz>-1$,
$$\mbox{$|Du^\ez|^{\bz+1}\in W^{1,2}_\loc(U)$ uniformly in $\ez>0$}$$
and
$$\lim_{\ez\to0}\|Du^\ez-Du\|_{L^q(V)}\to 0\quad\forall 1<q<\fz,\ \forall V\Subset U.$$
Thus,  $|Du|^{\bz+1}\in W^{1,2}_\loc(U)$ and {along a subsequence} $ |Du^\ez|^{\bz+1}\to |Du|^{\bz+1}$ {strongly} in
$ L^2_\loc(U)$ as $\ez\to0$.

 In view of   Definition \ref{def1.2}, the distributional Jacobi $ -\det D\big[ |Du^\ez |^\beta Du^\ez \big]$  is  well  defined.  By Lemma \ref{lem4.2}, one has
\begin{align}\label{xe5.1}
\lim_{\dz\to0} \int_U -\det D\big[(|Du^\ez|^2+\dz)^{\bz/2} Du^\ez\big]\psi\,dx
= \int_U -\det D\big[|Du^\ez|^\bz Du^\ez\big]\psi\,dx   \quad\forall \psi\in C^\fz_c(U).
\end{align}
By Lemma \ref{lem4.1} we know that
\begin{align}\label{xe5.2}
\lim_{\ez\to0}\int_U -\det D\big[|Du^\ez|^\bz Du^\ez\big]\psi\,dx=\int_U -\det D\big[|Du |^\bz Du \big]\psi\,dx
\quad\forall \psi\in C^\fz_c(U).
\end{align}

\medskip

{\bf Step 2.} By Lemma \ref{lem3.6} and  ${\bdz_\fz u^\ez} + \ez\bdz u^\ez =0$ in $U$, we have
\begin{align*}
\frac12[|D^2u^\ez|^2-(\bdz u^\ez)^2] |Du^\ez|^{2}
= |D^2uDu^\ez|^2-\bdz_\fz u^\ez \bdz u^\ez =|D^2uDu^\ez|^2+\frac1\ez(\bdz_\fz u^\ez)^2   \quad{\rm in}\ U.
\end{align*}

For $\dz>0$, by  this and Lemma \ref{lem3.7}, we obtain
\begin{align}  \label{xe5.3}
&-\det D\big[(|Du^\ez|^2+\dz)^{{\bz/2}}Du^\ez\big] \\
&\quad =\frac{1}{2}(|Du^\ez|^2+\dz)^{\bz }\big[|D^2u^\ez|^2-(\bdz u^\ez)^2\big]
+\bz (|Du^\ez|^2+\dz)^{{\bz-1}}\big[|D^2u^\ez Du^\ez|^2-\bdz u^\ez\bdz_\fz u^\ez\big] \nonumber\\
 &\quad\ge  (\bz+1) (|Du^\ez|^2+\dz)^{{\bz-1}}\big[|D^2u^\ez Du^\ez|^2
+\frac1\ez(\bdz_\fz  u^\ez )^2  \big] \nonumber\\
&\quad = \frac1{\bz+1}
  |D  (|D u^\ez|^2+\dz)^{(\bz+1)/2}|^2 +(\bz+1)\frac1\ez (|D u^\ez|^2+\dz)^{\bz-1} {(\bdz_\fz  u^\ez )^2 }
\quad{\rm in}\  U . \nonumber
  \end{align}

For any $0\le \psi\in C_c^\fz(U)$, by Lemma \ref{lem4.2} one has
\begin{align} \label{xe5.4}
&\int_U-\det D\big[ |Du^\ez |^ {\beta}Du^\ez \big] \psi\,dx\\
&\quad=\lim_{\dz\to0}\int_U-\det D\big[(|Du^\ez |^2+\dz)^{\beta/2}Du^\ez \big] \psi\,dx \nonumber  \\
&\quad \ge \frac1{\bz+1}\liminf_{\dz\to0}\int_U
  \Big[|D  (|D u^\ez|^2+\dz)^{(\bz+1)/2}  |^2 +(\bz+1)\frac1\ez (|D u^\ez|^2+\dz)^{\bz-1} {(\bdz_\fz  u^\ez )^2 }
\psi\Big]\,dx\nonumber\\
&\quad\ge \frac1{\bz+1} \int_U\Big[
  |D   |D u^\ez| ^{ \bz+1 }  |^2 +(\bz+1)\frac1\ez  |D u^\ez| ^{2\bz-2} {(\bdz_\fz  u^\ez )^2 }\psi \Big]\,dx, \nonumber
\end{align}
where, in the last inequality, we used that
$ (|D u^\ez|^2+\dz)^{(\bz+1)/2}\to |D u^\ez| ^{ \bz+1 } $ weakly in $W^{1,2}_\loc(U)$ as $\dz\to0$, and also that
 $(|D u^\ez|^2+\dz)^{\bz-1}  (\bdz_\fz  u^\ez )^2\to |D u^\ez| ^{2\bz-2} {(\bdz_\fz  u^\ez )^2 }$
 almost everywhere and have a dominant function $|D u^\ez| ^{2\bz+2} |D^2 u^\ez|^2\in L^1_\loc(U)$.

Letting $\ez\to0$, since  $|D u^\ez| ^{ \bz+1 } \to |Du|^{\bz+1}$ weakly in $W^{1,2}_\loc(U)$, we further have
\begin{align*}
\int_U-\det D[ |Du |^ {\beta}Du  ] \psi\,dx
&
\ge
\liminf_{\ez\to0 }
\frac1{\bz+1}\int_U |D|Du^\ez |^{\bz+1}|^2\psi\,dx\\
&\ge \frac1{\bz+1}\int_U |D|Du |^{\bz+1}|^2\psi\,dx\quad\forall 0\le \psi\in C_c^\fz(U),\nonumber
\end{align*}
 which gives the lower bound \eqref{xe1.7}.

\medskip

{\bf Step 3.} For any  $ \phi\in C_c^\fz(U)$ and $\dz\in[0,1)$, we have
\begin{align}\label{xe5.6}
\int_U-\det D\big[(|Du^\ez|^2+\dz)^{{\bz/2}}Du^\ez\big] \phi^2 \,dx
 &\le C   \frac{1+\bz^2}{1+ \bz} \int_U (|Du^\ez |^2+\dz)^{\bz+1}  [|\phi D^2\phi|+|D\phi|^2] \,dx.
  \end{align}
 Indeed,   by Lemma \ref{lem3.2} for $\dz\in(0,1]$  and Definition \ref{def1.2} for $\dz=0$, and by Young's inequality, one has
\begin{align*}
&\int_U-\det D\big[(|Du^\ez|^2+\dz)^{{\bz/2}}Du^\ez\big] \phi^2 \,dx\\
&\quad= -\frac12\int_U(|Du^\ez|^2+\dz)^{\bz } (D^2\phi^2 Du^\ez\cdot Du^\ez) \,dx
+\frac{1}{2\beta+2}\int_U(|Du^\ez|^2+\dz)^{\beta+1} \bdz\phi^2\,dx\nonumber\\
&\quad\quad-\frac{\beta}{2\beta+2}\int_U(|Du^\ez|^2+\dz)^{\frac{\beta-1}{2}}\big[D(|Du^\ez|^2+\dz)^{\frac{\bz+1}{2} }\cdot D u^\ez\big] (Du^\ez\cdot D\phi^2) \,dx\nonumber\\
 &\quad\le C( \frac12+\frac1{2+2\bz}+\frac{\bz^2}{2+2\bz})\int (|Du^\ez |^2+\dz)^{\bz+1}  \big[|\phi D^2\phi|+|D\phi|^2\big] \,dx\nonumber\\
&\quad\quad
+\frac{1}{ 2\beta+2}\int_U|D(|Du^\ez|^2+\dz)^{\frac{\bz+1}{2} }|^2 \phi^2  \,dx.\nonumber
 \end{align*}
Applying \eqref{xe5.3} for $\dz\in(0,1]$ and \eqref{xe5.4} for $\dz=0$, one has
$$\frac{1}{ 2\beta+2}\int_U|D(|Du^\ez|^2+\dz)^{\frac{\bz+1}{2} }|^2 \phi^2  \,dx\le \frac12
\int_U-\det D[(|Du^\ez|^2+\dz)^{{\bz/2}}Du^\ez] \phi^2 \,dx,$$
 and therefore,  we get   \eqref{xe5.6}.

{\bf Step 4.}  By \eqref{xe5.3} and \eqref{xe5.6}, we know that
\begin{align}
\label{xe5.7}\mbox{$0\le -\det D\big[(|Du^\ez |^2+\dz)^{\beta/2}Du^\ez \big] \in L^1_\loc(U)$ uniformly in $\dz\in(0,1]$.}
\end{align}
By \eqref{xe5.7}, \eqref{xe5.1}, and a density argument, we  know that
  $$ \lim_{\dz \to0}\int_U-\det D\big[(|Du^\ez |^2+\dz )^{\beta/2}Du^\ez \big]\psi\,dx$$
always exists
 for all $\psi\in C_c^0(U)$, and  is denoted by $\mu^\ez(\psi)$.
Moreover, $\mu^\ez$ is a nonnegative Radon measure, i.e.,  $0\le \mu^\ez \in \cm(U)$, and
$-\det D\big[(|Du^\ez |^2+\dz )^{\beta/2}Du^\ez \big] \,dx$ converges to $
\mu^\ez $  in the weak-$\star$ sense in $\cm(U)$ as $\dz\to0$.
Note that \eqref{xe5.1}  implies  that $\mu^\ez$ is induced by the distribution
$-\det D[ |Du^\ez |^ {\beta }Du^\ez ]$ uniquely.  Therefore, we can view
$-\det D[ |Du^\ez |^ {\beta }Du^\ez ]$ as the measure $\mu^\ez$.

Moreover,   by \   \eqref{xe5.6} with $\dz=0$,  given any subdomain $V\Subset U$, by a suitable choice of test function $\phi$ we also have
$$\|-\det D[ |Du^\ez |^ {\beta }Du^\ez ]\|(V)\le C\frac1{[\dist(V,\partial U)]^2} \frac{1+\bz^2}{1+ \bz} \int_W  |Du^\ez |  ^{2+2\bz}    \,dx,$$
where  $V\Subset W\Subset U$ with $\dist(W,\partial U)=\dist(V,\partial W)=\frac12\dist(V,\partial U)$.
Since $|Du^\ez |\in L^\fz (W)$  uniformly in
 $\ez \in (0,\ez_V)$ for some $\ez_V >0$,   we have
 $\|-\det D[ |Du^\ez |^ {\beta }Du^\ez ]\|(V)$ is  bounded uniformly in  $\ez \in (0,\ez_V)$.
Since
$$
0\le -\det D\big[ |Du^\ez |^ {\beta }Du^\ez \big]\in \cm(U),
$$
by \eqref{xe5.2} and a density argument we  know that
  $$ \lim_{\ez \to0}\int_U-\det D[ |Du^\ez | ^{\beta}Du^\ez ]\psi\,dx$$
always exists
 for all $\psi\in C_c^0(U)$, and is denoted by $\mu (\psi)$.   Moreover, one has that
$0\le \mu \in \cm(U)$ and
$-\det D[ |Du^\ez | ^{\beta}Du^\ez ]\,dx$ converges to  $\mu$ in the weak-$\star$ sense in $\cm(U)$ as $\ez\to0$.
By \eqref{xe5.2} we know that $\mu$ is induced by the distribution
 $-\det D[ |Du  | ^{\beta}Du  ]$ uniquely, and hence we can view  $-\det D[ |Du  | ^{\beta}Du  ]$ as $\mu$.

By the arbitrariness of $U$ we know that $-\det D[ |Du  | ^{\beta}Du  ]\in \cm(\Omega)$.
The upper bound \eqref{xe1.8} follows   from   \eqref{xe1.6} and  a suitable choice  of test functions $0\le \psi\in C^\fz_c(\Omega)$.
\end{proof}

\section{Proofs of  Theorems \ref{thm1.4} and  \ref{thm1.5}}                \label{sec6}

Given $p\in(1,\fz)$, let $u_p $ be any non-constant $p$-harmonic function  in a planar domain $\Omega$. For $\bz>-1$, one has $|Du_p|^\bz Du_p\in W^{1,2}_\loc(\Omega)$ and hence $-\det D[ |Du_p | ^{\bz }Du_p ]\in L^1_\loc(\Omega)$. Moreover,
we know that $E_{u_p}:=\{x\in\Omega, Du_p(x)=0\}$ is always discrete and hence is a null set,
 and $u\in C^\fz(\Omega\setminus E_{u_p})$. See  \cite{bi87,m88}.

\begin{lem}\label{lem6.1}
  Let  $\bz>-1$.
 Then
\begin{align}\label{xe6.1}
-\det D\big[ |Du_p | ^{\bz }Du_p \big]
& = \frac1{\bz+1}
  |D  |Du_p|^{\bz+1}  |^2+(\bz+1)(p-2) |Du_p|^{2\bz}\frac{(\bdz_\fz  u )^2 }{ |Du_p|^4 }
\quad{\rm in}\ \Omega\backslash E_{u_p}.
  \end{align}
 Consequently, if $p=2$, then
$$-\det D\big[ |Du_p | ^{\bz }Du_p \big]  =\frac1{\bz+1}|D|Du_p|^{\bz+1}|^2  \quad{\rm in}\ \Omega\backslash E_{u_p}; $$
if $p>2$, then
\begin{align}\label{xe6.2}
-\det D\big[ |Du_p | ^{\bz }Du_p \big]  \ge  \frac1{\bz+1}
  |D  |Du_p|^{\bz+1}  |^2
\quad{\rm in}\ \Omega\backslash E_{u_p};
  \end{align}
 if $1<p<2$, then
\begin{align}\label{xe6.3}\frac{p-1}{\bz+1} |D  |Du_p|^{\bz+1}  |^2
\le -\det D\big[ |Du_p | ^{\bz }Du_p \big]  \le  \frac1{\bz+1} |D  |Du_p|^{\bz+1}  |^2
    \quad{\rm in}\ \Omega\backslash E_{u_p}.
\end{align}

\end{lem}

\begin{proof}[Proof of Lemma \ref{lem6.1}]
In $\Omega\setminus E_{u_p}$,
applying  Lemma \ref{lem3.7}, we have
\begin{align*}
&-\det D\big[ |Du_p | ^{\bz }Du_p \big]
=\frac{1}{2} |Du_p | ^{2\bz }\big[|D^2 u_p |^2-(\bdz  u_p )^2\big]+
 \bz  |Du_p | ^{2(\bz-1)}\big[|D^2  u _p Du_p  |^2-\bdz  u _p \bdz_\fz  u \big].
 \end{align*}
Applying  \eqref{xe3.8}  to $u$, we have
\begin{align*}
\frac12[|D^2u_p|^2-(\bdz u_p)^2] =|Du_p|^{-2}[|D^2u_p Du_p|^2-\bdz_\fz u_p \bdz u_p]\quad{\rm in}\ \Omega\backslash E_{u_p},
\end{align*}
and hence
 \begin{align*}
-\det D[ |Du_p | ^{\bz }Du_p ]&=(\bz+1)
 |Du_p | ^{2(\bz-1)} [|D^2  u_p  Du_p  |^2-\bdz  u_p  \bdz_\fz  u_p ] \quad{\rm in}\ \Omega\backslash E_{u_p}.
  \end{align*}

Note
\begin{align*}
\bdz u_p=-
(p-2)\frac{\bdz_\fz u_p}{|Du_p|^2} \quad{\rm in}\ \Omega\backslash E_{u_p}.
\end{align*}
For $\bz>-1$, one gets {\eqref{xe6.1}.}
  When $2<p<\fz$, \eqref{xe6.1} gives {\eqref{xe6.2}}.
When $1<p<2$ and $\bz>-1$,  since
$$ |D^2  u_p  Du_p  |^2-\bdz  u  \bdz_\fz  u  =|D^2  u_p  Du_p  |^2+(p-2)  |Du_p | ^{-2} (\bdz_\fz  u_p )^2
\ge (p-1)|D^2  u _p Du_p  |^2 \quad{\rm in}\ \Omega\backslash E_{u_p},$$
one has
$$-\det D\big[ |Du_p | ^{\bz }Du_p \big]\ge
    \frac{p-1}{\bz+1} |D  |Du_p|^{\bz+1}  |^2 \quad{\rm in}\ \Omega\backslash E_{u_p} $$
    as desired.
\end{proof}

\begin{lem}\label{lemma6.2} For any $\phi\in C_c^\fz(\Omega)$, one has
\begin{align}\label{xe6.6}
 \int_\Omega-\det D\big[|Du_p|^{\beta}Du_p\big] \phi^2\,dx
&\le C \Big[1+ \frac{1}{1+\bz} + \frac1{p-1}\frac{\bz^2}{\bz+1} \Big]
\int_\Omega|Du_p|^{2\beta+2} \big[|\phi D^2\phi |+|D\phi|^2\big]  \,dx.
 \end{align}
\end{lem}

\begin{proof}
For all $\phi\in C_c^\fz(\Omega)$, write
\begin{align*}
\int_\Omega-\det D\big[|Du_p|^{\beta}Du_p\big] \phi^2\,dx
&=-\frac12\int_\Omega|Du_p|^{2\beta} (D^2\phi^2 Du_p\cdot Du_p) \,dx
+\frac{1}{2\beta+2}
\int_\Omega|Du_p|^{2\beta+2} \bdz\phi^2 \,dx\nonumber\\
&\quad-\frac{\beta}{{ \beta+1}}
\int_\Omega\big[D |Du_p| ^{\bz+1 }\cdot D u_p\big]
(Du_p\cdot D\phi^2)|Du_p|^{\beta-1} \,dx\nonumber\\
&=:I_1+I_2+I_3.
\end{align*}

Clearly,
\begin{align*}
I_1+I_2
&\le C (1+\frac1{1+\beta})\int_\Omega |Du_p|^{2+2\bz}\big[|\phi D^2\phi|+|D\phi|^2\big] \,dx.
\end{align*}
When $1<p\le 2$, by Young's inequality and \eqref{xe6.3} one has
\begin{align*}
  I_3
&\le \frac{p-1}{2+2\bz}\int_\Omega  |D |Du_p| ^{\bz+1 }|^2\phi^2\,dx+
\frac{4\beta^2}{(2+2\beta)(p-1)}
\int_\Omega  |Du_p|^{2+2\beta} |D\phi|^2\,dx\\
&\le \frac12 \int_\Omega-\det D\big[|Du_p|^{\beta}Du_p\big] \phi^2\,dx+  C
\frac{\beta^2}{(1+ \beta) (p-1)}
\int_\Omega |Du_p|^{2+2\beta} |D\phi|^2\,dx.
\end{align*}
When $2<p\le 4$, by \eqref{xe6.1}  similarly to the case $1<p\le 2$ one also has
\begin{align*}
  I_3
&\le \frac12 \int_\Omega-\det D\big[|Du_p|^{\beta}Du_p\big] \phi^2\,dx+  C
\frac{\beta^2}{(1+ \beta) (p-1)}
\int_\Omega |Du_p|^{2+2\beta} |D\phi|^2\,dx.
\end{align*}
When $p\ge 4$, by Young's inequality and \eqref{xe6.1}, one has
\begin{align*}
I_3
&\le \bz
\int_\Omega|\Delta_\fz u _ p\phi D\phi| |Du_p|^{2\beta-1} \,dx\\
&\le \frac{(p-2)(\bz+1)}2
\int_\Omega|\Delta_\fz u_p|^2 |Du_p|^{2\beta-4}\phi^2\,dx+\frac{8|\beta|^2}{(p-2)(\bz+1)} \int_\Omega | D\phi|^2|Du_p|^{2\beta+2} \,dx\\
&\le  \frac12 \int_\Omega-\det D[|Du_p|^{\beta}Du_p] \phi^2\,dx+  C
\frac{\beta^2}{(1+ \beta) (p-1)}
\int_\Omega |Du_p|^{2+2\beta} |D\phi|^2\,dx.
\end{align*}
We therefore obtain   \eqref{xe6.6}.
\end{proof}

\begin{proof}[Proof  of Theorem {\ref{thm1.4}}]
Theorem {\ref{thm1.4}} follows immediately from  Lemmas \ref{lem6.1} and \ref{lemma6.2}.
\end{proof}
\begin{proof}[Proof  of Theorem {\ref{thm1.5}}]
Given any bounded smooth domain $\Omega $ and $g\in \lip(\partial\Omega)$
for $1<p\le \fz$ denote by  $u_p$ the unique $p$-harmonic functions in $\Omega $ with boundary $g$.
Moreover $u_p\to u_\fz$ in $C^{0,\alpha}(\overline \Omega)$  for any $\alpha\in(0,1)$ and weakly in $W^{1,q}(\overline\Omega) $  for any $ 1<q<\fz$ as $p\to\fz$.  This is well known; see for example \cite{lM}.
For reader'{s} convenience, a proof is given below.
Since $u_\fz$  is the absolute minimizer with boundary $g$, we know that
$\|Du\|_{L^\fz(\Omega)}= \|g\|_{\lip(\partial\Omega)}$.
Moreover we may extend $g$  to $\Omega$ with the same Lipschitz norm via   the Mcshane extension.
For $1<p<\fz$, since $u_p$ is the minimizer and to see that
$$\|Du_p\|_{L^p(\Omega)}\le \|D  g\|_{L^p(\Omega)}\le  |\Omega|^{1/p}  \|g\|_{\lip(\partial\Omega)}. $$
Given any $1< q<\fz$, for $p>q$  by H\"older inequality we know that
 $\|Du_p\|_{L^q(\Omega)}\le  |\Omega|^{ 1/q }  \|g\|_{\lip(\partial\Omega)}$  and hence uniformly bounded.    Then $u_p\in C^{0,1-n/q}(\overline \Omega)$ for large $p>q$ uniformly. Thus
 $u_p$ converges to $u$ in $C^{0,1-n/q}(\overline\Omega)$ as $p\to\fz$ up to some subsequence.
Since $u_p$ is also {a} viscosity solution to $\Delta_\fz v+ \frac1{p-2}\Delta v|Dv|^2=0$ in $\Omega$ with boundary $g$, by the compactness of viscosity solution we see that $u$ is a viscosity solution to $\Delta_\fz v=0$.
Observing that $u$ and $u_\fz$ {satisfy} the same boundary {condition}, by {Jensen's} uniqueness {result,} we know that $u=u_\fz$.
Since $Du_p\in {L^q(\Omega)}$ for $p>q$ uniformly, we know that $Du_p$
converges weakly to $Du_\fz$ in $L^q(\Omega)$ as $p\to\fz$.

   Since $Du_p\in L^2(\Omega)$ uniformly in $p$, by Lemma \ref{lemma6.2}, we know that
$-\det  D^2u_p  \in L^1_\loc(\Omega)$ uniformly in $p>2$.
 By Lemma \ref{lem6.1}, $D|Du_p|\in L^2_\loc(\Omega)$  uniformly in $p>2$.
By the Sobolev embedding theorem, we know that  $|Du_p|$ converges to some function $h$ in $L^q_\loc
(\Omega)$
and weakly in $W^{1,2}_\loc(\Omega) $ as $ p\to\fz$.
By building up a flatness estimates similarly to  \cite[Lemma 2.7]{kzz} (here we omit the details; see \cite{ll}),  one has $h=|Du_\fz|$.
Since $Du_p\to Du_\fz$ weakly in $L^q(\Omega)$ with $q>1$, we deduce that
 $Du_p\to Du_\fz$ in $L^q_\loc(\Omega)$ as $p\to\fz$.

Applying Lemma \ref{lem4.1} we know that  $ |Du_p|^\beta Du_p\to |Du_\fz|^\beta Du_\fz$ in   $L^{  q}_\loc(\Omega)$ for any $1<q<\fz$ and moreover,
\begin{align}\label{xe6.7}\int_\Omega-\det D\big[|Du_\fz|^\beta Du_\fz\big]\psi\,dx=\lim_{p\to\fz}\int_\Omega-\det D\big[|Du_p|^\beta Du_p\big]\psi\,dx\quad \forall \psi\in C^\fz_c(\Omega).\end{align}
Lemma \ref{lemma6.2} yields that  $-\det D[|Du_p|^\beta Du_p]\in L^1(\Omega)$ uniformly in $p>2$.
By a density argument, we know that \eqref{xe6.7} holds for all  $\psi\in C^0_c(\Omega)$,
i.e.,
 $-\det D [ |Du_p|^\beta Du_p] \to -\det D [ |Du_\fz|^\beta Du_\fz]$  in the weak$-\star$ sense in $\cm(\Omega)$.
\end{proof}

\begin{rem}\rm
One could also prove Theorem \ref{thm1.3} via Lemma \ref{lem6.1}, Lemma \ref{lemma6.2},   and \eqref{xe6.7}.
\end{rem}

\renewcommand{\thesection}{Appendix A}
 \renewcommand{\thesubsection}{ A }
\newtheorem{lemapp}{Lemma \hspace{-0.15cm}}
\newtheorem{corapp}[lemapp] {Corollary \hspace{-0.15cm}}
\newtheorem{remapp}[lemapp]  {Remark  \hspace{-0.15cm}}
\newtheorem{defnapp}[lemapp]  {Definition  \hspace{-0.15cm}}
\renewcommand{\theequation}{A.\arabic{equation}}

\renewcommand{\thelemapp}{A.\arabic{lemapp}}

\section{Some sharpness in the plane}

 At the borderline case $\bz=-1$,   we have the following result, which will be used later.
\begin{lemapp}\label{lem-log}
Let $1<p<\fz$. If $u_p$ is non-constant $p$-harmonic function  in a   domain $\Omega\subset\mathbb R^2$, then
\begin{align*}
\Big|D\big[|Du_p|^{-1}Du_p\big]\Big|^2
=|D\log|Du_p||^2
 + (p-2)p \frac{(\bdz_\fz u_p)^2}{|Du_p|^4}  \ a.\,e. \nonumber
\end{align*}
In particular, if $p=2$,
$$
\Big|D\big[|Du_p|^{-1}Du_p\big]\Big|^2
= |D\log|Du_p||^2  \ a.\,e. ;$$
if $ p>2$,
$$
\Big|D\big[|Du_p|^{-1}Du_p\big]\Big|^2
\ge |D\log|Du_p||^2   \ a.\,e.;$$
and if $1<p<2$,
\begin{align*}
 |D\log|Du_p||^2 \ge  \Big|D\big[|Du_p|^{-1}Du_p\big]\Big|^2
&
\ge {(p-1)^2}  |D\log|Du_p||^2   \ a.\,e.
\end{align*}
\end{lemapp}

\begin{proof}
In $\Omega\backslash E_{u_p}$, one has
$$|D\log|Du_p||^2= \frac{|D ^2u_pDu_p|^2}{|Du_p|^4}
\le \frac{|D^2u_p|^2}{|Du_p|^2}  $$
and
$$
\Big|D\big[|Du_p|^{-1}Du_p\big]\Big|^2
=| |Du_p|^{-1}D^2u_p- |Du_p|^{-3} D^2u_p Du_p\otimes Du_p|^2 =\frac{|D^2u_p|^2}{|Du_p|^{2}}-\frac{|D^2u_pDu_p|^2}{|Du_p|^4} .
$$

Recall that
$$
   { |D^2u_p Du_p^2} -  {\Delta u_p \Delta_\fz u_p } =\frac12
   \big[|D^2u_p|^2-(\Delta u_p)^2\big]|Du_p|^2 \quad  \mbox{in $\Omega\backslash E_{u_p}$}.
   $$
Replacing $\Delta u_p$ with $(p-2)\frac{\Delta_\fz u_p}{|Du_p|^2}$,
  one has
$$
 \frac12\big[|D^2u_p|^2-(p-2)^2\frac{(\bdz_\fz u_p)^2}{|Du_p|^4} \big]
 =\frac{|D^2u_p Du_p|^2}{|Du_p|^2} -(p-2)\frac{(\bdz_\fz u_p)^2}{|Du_p|^4},
$$
or equivalently,
$$ |D^2u_p|^2 = 2\frac{|D^2u_p Du_p|^2}{|Du_p|^2}
 + (p-2)p \frac{(\bdz_\fz u_p)^2}{|Du_p|^4}.
$$
Thus,
$$
\Big|D\big[|Du_p|^{-1}Du_p\big]\Big|^2
=\frac{|D^2u_p Du_p|^2}{|Du_p|^4}
 + (p-2)p \frac{(\bdz_\fz u_p)^2}{|Du_p|^6}  =\frac12 \frac{|D^2u_p |^2}{|Du_p|^2}
 + \frac{(p-2)p}2 \frac{(\bdz_\fz u_p)^2}{|Du_p|^6}.
$$

When $p=2$, then
$$
\Big|D\big[|Du_p|^{-1}Du_p\big]\Big|^2
=\frac12\frac{|D^2u_p |^2}{|Du_p|^2}= \frac{|D^2u_pDu_p|^2}{|Du_p|^4} \quad \mbox{in $\Omega\backslash E_u$};$$
When $p>2$,  we have
$$
\Big|D\big[|Du_p|^{-1}Du_p\big]\Big|^2
\ge \frac12 \frac{|D^2u_p|^2}{|Du_p|^{2}}\ge \frac{|D^2u_p Du_p|^2}{|Du_p|^2},
$$
while when  $1<p<2$,
 we have
$$
\frac12\frac{|D^2u_p|^2}{|Du_p|^2}
\ge  \Big|D\big[|Du_p|^{-1}Du_p\big]\Big|^2
\ge \frac{{(p-1)^2} }{2}\frac{|D^2u_p|^2}{|Du_p|^2}  \ge \frac {(p-1)^2}2 \frac{|D^2u_p Du_p|^2}{|Du_p|^4}  .
$$
\end{proof}

For   $1<p<\fz$, we recall the extremal $p$-harmonic function  constructed by \cite[Section 7]{im89}.
 Here we keep notation {the} same as therein.
 Let \begin{equation}\label{xea.1}H(\xi)=\left(\frac{\xi}{|\xi|}+\ez \frac{|\xi|^3}{\xi^3} \right)|\xi|^{1/d} \quad\forall \xi\in\mathbb C
 \end{equation}
 with $$\frac1d= \frac12\left(-p+\sqrt{16(p-1)+(p-2)^2} \right)>0\quad \mbox{and}\quad \ez =\frac{1-d}{1+3d}.$$
If $p=2$, then $d=1$ and $\ez=0$ and hence $H(\xi)=\xi$. If $p\ne 2$, then $d>0$ and $\ez\ne0$, and $H$ is a quasiconformal homeomorphism on the whole plane.
According to \cite[Theorem 2 ]{im89}, $H(\xi)$ satisfies  \cite[(18)  with $n=1$]{im89}, that is,
\begin{equation}\label{xea.2} H_{\bar \xi}=\left(\frac12-\frac1p\right)
\left[ \frac{\xi}{\bar \xi } H_\xi+\frac{\bar \xi}{ \xi }\bar {H_\xi} \right], \end{equation}
where $H_\xi=\frac12(H_x-iH_y)$ and $H_{\bar \xi}=\frac12(H_x+iH_y)$ for $\xi=x+iy$.

Let $f(z)$ denote  the inverse of $H(\xi)$ in $\mathbb C$, so that $ f(H(\xi))=\xi$ and $H(f(z))=z$ for all $z,\xi\in\mathbb C$.
From \eqref{xea.2}, one deduces
$$ f_{\bar z}=\left(\frac1p-\frac12\right)\left[ \frac{f}{\bar f } \bar{f_z}+\frac{\bar f}{ f }f_z \right],$$
where $f_z=\frac12(f_x-if_y)$ and $f_{\bar z}=\frac12(f_x+if_y)$ for $z=x+iy$.
This then defines a $p$-harmonic function  $w$ in the whole plane
so that its complex derivative $w_z=f.$

We have the following properties.

\begin{lemapp}\label{lema.2}  One has   $\log|Dw|=\log|f|\notin W^{1,2}_\loc(\rr^2)$ and $|Dw|^{-1}Dw\notin W^{1,2}_\loc(\rr^2)$.
\end{lemapp}

\begin{proof}[Proof of Theorem \ref{lema.2}]
By Lemma  \eqref{lem-log}, we only need to prove
$\log|f|\notin W^{1,2}_\loc(\rr^2)$. We argue by contradiction.
 Assume that  $\log|f|\in W^{1,2}_\loc(\mathbb{C})$.
Note that $f(tz)=t^df(z)$ for any $t\ge 0$. A direct calculation implies that
\begin{equation}\label{xea.3}
D|f|(z)=t^{1-d}D|f|(tz),\quad D\log|f|(z)=\frac{D|f|(z)}{|f(z)|},\quad z\in \mathbb{C}\backslash\{f^{-1}(0)\},\ t\ge0,
\end{equation}
where $f^{-1}(0)=\{z\in \mathbb{C}: f(z)=0\}$. For each $R>0$, we know that $f^{-1}(0)\cap \{z\in \mathbb{C},|z|<R\}$ is discrete, and from $f(tz)=t^df(z)$, \eqref{xea.3} we conclude that
\begin{align*}
&\int_{|z|<R}|D\log|f|(z)|^2\,dz=\int_{|z|<R}\frac{|D|f|(z)|^2}{|f(z)|^2}\,dz
=\int_{|\xi|<tR}\frac{|D|f|(\frac{\xi}{t})|^2}{|f(\frac{\xi}{t})|^2}d\frac{\xi}{t}\\
&=\int_{|\xi|<tR}\frac{(\frac{1}{t})^{2d-2}|D|f|(\xi)|^2}{
(\frac{1}{t})^{2d}|f(\xi)|^2} \frac{1}{t^2} \,d\xi
=\int_{|\xi|<tR}|D\log |f| (\xi)|^2\,d\xi,\quad \forall t>0.
\end{align*}
Letting $t\to 0$ we conclude that
$ D \log |f|(z)= 0$ whenever $|z|<R$, and hence, by the arbitrariness of $R$,  for all $z\in\mathbb C$. Thus $|f|$ is a positive constant in the whole plane. This contradicts that $f(tz)=t^df(z)$ for all $t>0$ and $z\in\mathbb C$, where {we} recall that $d>0$.
\end{proof}

\begin{lemapp}\label{lema.3} One has \begin{equation}\label{xea.4}
\sup_{\mathbb C\setminus\{0\}} \frac{|f_{\bar z}|} {|f_{ z}|}=  \frac{|p-2|}p =\frac{K(p)-1}{K(p)+1} \quad\mbox{with }\quad
 K(p)=\max\left\{\frac1{p-1},p-1\right\}.\end{equation}
 In general, for $\bz>-1$,  writing $g=|f|^\bz f$, one has
 \begin{equation}\label{xea.5}\sup_{\mathbb C\setminus\{0\}} \frac{|g_{\bar z}|}{|g_z|}= \frac{K(p,\bz)-1}{K(p,\bz)+1} \quad\mbox{with }\quad
 K(p,\bz)=\max\left\{ \frac{p-1}{\bz+1},\frac{\bz+1}{p-1},
\bz+1,\frac1{\bz+1} \right\} .\end{equation}

\end{lemapp}

\begin{proof} Since  $H(\xi)$ is the inverse of $f(z)$, \eqref{xea.4} is equivalent to   $$\sup_{\mathbb C\setminus\{0\}} \frac{|H_{\bar \xi}|}{|H_{  \xi}|}= \frac{|p-2|}p.$$
 We already have  $$ |H_{\bar \xi}| =\frac12\left|\frac{\xi}{\bar \xi } H_\xi+\frac{\bar\xi}{ \xi } \bar H_\xi\right |\le\frac{|p-2|}p |H_{  \xi}| \quad \mbox{in $\mathbb C\setminus\{0\}$}.$$
 Taking the derivative $\partial_\xi$ on both sides of \eqref{xea.1}

one has  \begin{align*}H_\xi
 &= \frac12|\xi|^{1/d-1}\left[\left(\frac1d+1\right) + \left (\frac1d-3\right)\ez\frac{|\xi|^4}{\xi^ 4 }\right].
 \end{align*}
If $\xi\in\rr$, then    $H_\xi(\xi)\in\rr$, and hence, \eqref{xea.1} gives $H_{\bar \xi}=\frac{p-2}p H_\xi$
 as desired.

Next, for $\bz>-1$,
write $g=|f|^{\bz}f$ and $G=g^{-1}$. By \cite[Section 3]{bi87} one has
$$g_{\bar z}=-\frac12\left(\frac{p-2-\bz}{p+\bz}+\frac {\bz}{\bz+2}\right)\frac{\bar g}{g}g_z -\frac12\left(\frac{p-2-\bz}{p+\bz}-\frac {\bz}{\bz+2}\right)\frac{ g}{\bar g}\bar g_z $$
and hence
\begin{equation}\label{xea.6}G_{\bar \xi}= \frac12\left(\frac{p-2-\bz}{p+\bz}+\frac {\bz}{\bz+2}\right)\frac{\bar \xi}{\xi}G_\xi +\frac12\left(\frac{p-2-\bz}{p+\bz}-\frac {\bz}{\bz+2}\right)\frac{ \xi}{\bar \xi}\bar G_\xi.
\end{equation}
Thus,
$$
\sup_{\mathbb C\setminus\{0\}}\frac{|G_{\bar\xi}|}{|G_\xi|}\le  \max\left\{ \frac{|p-2-\bz|}{p+\bz},  \frac {|\bz|}{\bz+2} \right\} =\frac{K(p,\bz)-1}{K(p,\bz)+1}
$$
with  $K(p,\bz)$ be as in \eqref{xea.5}.
Moreover, note that
$$G(\xi)=H(|\xi|^{-\bz/(\bz+1)}\xi)=\left (\frac{\xi}{|\xi|}+\ez \frac{|\xi|^3}{\xi^3}  \right) |\xi|^{-1/(\bz+1)d} \quad\forall \xi\in\mathbb C,$$
and hence
$$G_\xi(\xi)
 = \frac12|\xi|^{1/(\bz+1)d-1}\left[\left(\frac1{(\bz+1)d}+1\right) + \left (\frac1{(\bz+1)d}-3\right)\ez\frac{|\xi|^4}{\xi^{4}}\right].$$
 If $ \frac{|p-2-\bz|}{p+\bz} \ge  \frac {|\bz|}{\bz+2} $, for $\xi\in\rr$  we have   $ G_\xi(\xi)\in\rr$, and therefore,  \eqref{xea.6} gives
 $G_{\bar \xi}(\xi)=   \frac{p-2-\bz}{p+\bz}  G_\xi(\xi) $ as desired.
 If $\frac{|p-2-\bz|}{p+\bz}<\frac {|\bz|}{\bz+2}$, for $\xi\in\mathbb R-i\rr$
 we have $\bar \xi=i\xi$, $G_\xi(\xi)\in\rr$ and
  $$\frac\xi{\bar \xi}G_\xi(\xi) = -\frac{\bar \xi}\xi \bar G_\xi(\xi)= -iG_\xi(\xi),$$
  which together with \eqref{xea.6} gives
    $G_{\bar \xi}(\xi)=   \frac{ \bz}{2+\bz}  iG_\xi(\xi) $
   as desired.
\end{proof}

Lemma \ref{lema.3} gives the sharpness of   constants in \eqref{xe1.9}.
\begin{remapp} \label{remA.4}\rm By a standard calculation,   \eqref{xea.4} gives
$$\esup_{\mathbb C}\frac{2\big[|f_z|^2+|f_{\bar z}|^2\big]}{|f_z|^2-|f_{\bar z}|^2}=\frac{(p-1)^2+1}{p-1}.$$
Since  $|D^2w|^2=2[|f_z|^2+|f_{\bar z}|^2] $ and $-\det D^2w=|f_z|^2-|f_{\bar z}|^2$, we write this as
$$\esup_{\mathbb C}\frac{|D^2w|^2}{-\det D^2w}=\frac{(p-1)^2+1}{p-1}=(p-1)+\frac1{p-1}.$$
Thus the constant in \eqref{xe1.9} is sharp. Note that $(p-1)+\frac1{p-1}$ converges to $\fz$ as $p\to\fz$.

For $\bz>-1$, in a similar way, \eqref{xea.5} gives
$$  \esup \frac{|D[|Dw|^\bz Dw]|^2}{-\det D[|Dw|^\bz Dw]}=
\frac{K(p,\bz)^2+1}{K(p,\bz)}=K(p,\bz)+ \frac{1}{K(p,\bz)},$$
 and hence, the  constant  in \eqref{xe1.9} is   sharp.
 We also note that $K(p,\bz)+ \frac{1}{K(p,\bz)}$ converges to $\fz$ as $p\to\fz$.
 \end{remapp}

\section*{Acknowledgment}
Y. Zhou would like to thank Professor Juan J. Manfredi  and Professor Nageswari Shanmugalingam for their kind suggestions and comments
 on the best constants as in Lemma \ref{lema.3} and Remark \ref{remA.4}.

\bigskip\bigskip

\noindent Hongjie Dong

\noindent Division of Applied Mathematics, Brown University, Providence RI 02912, USA

\noindent{\it E-mail }:  \texttt{Hongjie\_Dong@Brown.edu}

\bigskip
\noindent Fa Peng

\noindent
Academy of Mathematics and Systems Science, the Chinese Academy of Sciences, Beijing 100190, P. R. China

\noindent{\it E-mail }:  \texttt{fapeng@amss.ac.cn}

\bigskip

\noindent Yi Ru-Ya Zhang

\noindent
Academy of Mathematics and Systems Science, the Chinese Academy of Sciences, Beijing 100190, P. R. China

\noindent{\it E-mail }:  \texttt{yzhang@amss.ac.cn}
\bigskip

\noindent Yuan Zhou

\noindent School of Mathematical Sciences, Beijing Normal University, Haidian District Xinejikou Waidajie No.19, Beijing 100875, P. R. China

\noindent {\it E-mail }:
\texttt{yuan.zhou@bnu.edu.cn}

\end{document}